%% file: paperCHNSanalysis.tex
\documentclass[review,onefignum,onetabnum]{siamart171218}

\input{shared}

\ifpdf
\hypersetup{
  pdftitle={Numerical analysis of a discontinuous Galerkin method for Cahn--Hilliard--Navier--Stokes equations},
  pdfauthor={Chen Liu}
}
\fi

\newcommand{\bfzeta}{\boldsymbol{\zeta}}
\newcommand{\bfxi}{\boldsymbol{\xi}}

\begin{document}

\maketitle
\hspace{5.75cm}{July~07$^\mathrm{th}$, 2018}\\

\begin{abstract}
In this paper, we derive a theoretical analysis of an interior penalty discontinuous Galerkin methods for solving the Cahn--Hilliard--Navier--Stokes model problem. We prove unconditional unique solvability of the discrete system, obtain unconditional discrete energy dissipation law, and derive stability bounds with a generalized chemical energy density. Convergence of the method is obtained by proving optimal a priori error estimates. Our analysis of the unique solvability is valid for both symmetric and non-symmetric versions of the discontinuous Galerkin formulation.
\end{abstract}



\section{Introduction}
The Cahn--Hilliard--Navier--Stokes system
strikes an optimal balance in terms of thermodynamical rigor and computational efficiency for modeling two-component binary flow.
The model that belongs to the class of diffusive interphase or phase-field models, attracts much attention in physics, chemistry, biology, and engineering fields. In recent years, driven by the major developments of numerical algorithms and by increased availability of computational resources and capabilities, the direct numerical simulation of Cahn--Hilliard--Navier--Stokes equations has become increasingly popular. The spectrum of applications for this model involves modeling spinodal decomposition \cite{kay2007efficient}, transport processes in porous media \cite{FLABR2018Direct}, and wetting phenomenon \cite{alpak2016phase,FLSAR2017energy}. 
\par
This paper is devoted to the numerical analysis of an  interior penalty discontinuous Galerkin method for the Cahn--Hilliard--Navier--Stokes equations. We prove, in two and three dimensions, the unconditional unique solvability, obtain an unconditional energy dissipation law, and derive stability bounds with generalized chemical energy density function. Convergence of the method is obtained by proving optimal a priori error estimates. The main contributions of this paper are the unique solvability proof and the error estimates. In addition, the technique we use for proving the existence and uniqueness is valid for both symmetric and non-symmetric interior penalty discontinuous Galerkin formulations.
\par
Over the last ten years, the convergence analysis for the Cahn--Hilliard--Navier--Stokes model has been extensively investigated, for numerical schemes based on the continuous finite element method. In \cite{feng2006fully}, continuous $\mathbb{P}_2-\mathbb{P}_0$ elements are used for the approximation of the velocity and pressure whereas continuous $\mathbb{P}_r$ elements, for $r\geq 1$ are used for the approximation of the chemical potential and order parameter. Convergence of the solution is obtained via a compactness argument.   
Kay, Styles, and Welford in \cite{kay2008finite} analyzed semi-discrete and fully discrete finite element schemes in two-dimensions. Under a CFL-like condition, they obtained  a priori error estimates for the semi-discrete method and a convergence proof based on a compactness argument for the fully discrete scheme.  
Diegel, Wang, Wang, and Wise in \cite{diegel2017convergence} analyzed a second order in time mixed finite element method, based on Crank--Nicolson method. Continuous $\mathbb{P}_r$ elements are used for the chemical potential, order parameter and pressure whereas continuous $\mathbb{P}_{r+1}$ are
used for the velocity with any positive integer $r$.  The work contains unconditional energy stability and optimal error estimates. 
In \cite{han2015second,CaiShen18}, a projection method is used to handle the Navier--Stokes equations. Han and Wang introduce a second order in time method and show unconditional unique solvability of the algorithm.  The work \cite{han2015second} does not contain any theoretical proof of convergence of the solution.  Cai and Shen obtain unconditional unique solvability, derive error estimates and show a convergence analysis based on a compactness argument. In \cite{CaiShen18}, both chemical potential and order parameter are approximated by continuous $\mathbb{P}_2$ elements and the velocity and pressure are approximated by a stable pair of finite element spaces.
In addition of using continuous finite elements in space, all the works mentioned above assume a special form of chemical energy density, namely a double-well potential, also refered to as Ginzburg--Landau potential.  The coupling term in the momentum equation of the Navier--Stokes system may take
several forms, that yield different numerical methods and impact their analysis.  We note that 
in \cite{feng2006fully,diegel2017convergence,CaiShen18}, the coupling term is the product of the chemical potential and the gradient of  the order parameter.  In the other works \cite{kay2008finite,han2015second} as well as in our present work, the coupling term is the product of the order parameter and the gradient of the chemical potential. 
\par
To the best of our knowledge, the theoretical analysis for fully discrete interior penalty discontinuous Galerkin scheme of the Cahn--Hilliard--Navier--Stokes system is not yet available in the literature.
While there are few works on the numerical analysis of the coupled Cahn--Hilliard and Navier--Stokes equations, the literature on numerical methods for solving the Cahn--Hilliard equation (resp. the Navier--Stokes equations) is abundant.  Finite element methods and interior penalty discontinuous
Galerkin methods have been employed for each equation separately.  We refer the reader to \cite{feng2007fully,kay2009discontinuous,aristotelous2013mixed,Eyre1998EyreScheme} for the error analysis of Cahn--Hilliard equations and to \cite{temam2001navier,GirRav86,GuermondMinevShen,girault2005splitting,girault2005discontinuous} for the Navier--Stokes equations, and the references herein for a non-exhaustive list.
\par
The outline of the paper follows. The mathematical model and related analytical properties are described in \cref{sec:CHNS:model}. The numerical method and analysis, including the proof of unique solvability, stability analysis, and error analysis are addressed in \cref{sec:CHNS:numerical_analysis}. Conclusions are given in the last section.

\section{Mathematical Model}\label{sec:CHNS:model}
Let $\Omega\subset\IR^d$, where $d=2$ or $3$, be an open bounded polyhedral domain and $\normal$ denote the outward normal of $\Omega$. The unknown variables in Cahn--Hilliard--Navier--Stokes equations are the order parameter $c$, the chemical potential $\mu$, the velocity $\vec{v}$ and the pressure $p$, satisfying:
\begin{subequations}\label{eq:CHNS:model}
\begin{align}
\partial_t c - \Delta \mu  + \div{(c\vec{v})}&= 0, && \text{in}~(0,\, T)\times\Omega, \label{eq:CHNS:model:a}\\
\mu &= \Phi'(c) - \kappa\, \laplace c, && \text{in}~(0,\, T)\times\Omega,\label{eq:CHNS:model:b}\\
\partial_t \vec{v} + \vec{v}\cdot\grad{\vec{v}} - \mu_\mathrm{s}\laplace{\vec{v}} &= -\grad{p} - c\grad{\mu}, && \text{in} ~(0,\, T)\times\Omega, \label{eq:CHNS:model:c}\\
\div{\vec{v}} &= 0, && \text{in} ~(0,\, T)\times\Omega. \label{eq:CHNS:model:d}
\end{align}
Equations~\eqref{eq:CHNS:model:a} and~\eqref{eq:CHNS:model:b} represent the mass conservation equations for two components. The parameter $\kappa$ is a positive constant, which is related to the thickness of the interface between the two phases. The function $\Phi$ is a scalar potential function, also called chemical energy density.
Equations~\eqref{eq:CHNS:model:c} and ~\eqref{eq:CHNS:model:d} are the momentum and incompressibility equations respectively. The parameter $\mu_\mathrm{s}$ is the fluid viscosity.
For our model problem, the following boundary and initial conditions are added:
\begin{align}
\grad c\cdot\normal &= 0, && \text{on}~(0,\, T)\times\partial\Omega,\label{eq:CHNS:model:out_c}\\
\grad\mu \cdot\normal &= 0, && \text{on}~(0,\, T)\times\partial\Omega,\label{eq:CHNS:model:normalFluxMu} \\
\vec{v} &= \vec{0}, && \text{on}~(0,\, T)\times\partial\Omega,\label{eq:CHNS:model:boundary_v}\\
c &= c^0, && \text{in}~\{0\}\times\Omega, \label{eq:CHNS:model:initial_c}\\
\vec{v} &= \vec{v}^0, && \text{in}~\{0\}\times\Omega. \label{eq:CHNS:model:initial_v}
\end{align}
The pressure $p$ is uniquely defined up to an additive constant, to close this system, we also assume the mean pressure on $\Omega$ is zero:
\begin{equation}\label{eq:CHNS:model:avgP0}
\int_\Omega p = 0.
\end{equation}
\end{subequations}
\begin{remark}
The order parameter $c$ can either be a volume or a mass fraction of one of the two components $c_1,~c_2$ or the difference between mass fractions. In the former case, for instance $c = c_1$\,, from the definition of the fraction it is straightforward to see $c \in [0,1]$\,. In the latter case, for instance $c= c_1 - c_2$\,, due to the constraint $c_1+c_2=1$ we have $c \in [-1,1]$\,.
\end{remark}
\begin{remark}
Under the assumption of the incompressibility constraint \eqref{eq:CHNS:model:d}, it is possible to consider employing advection operator $\vec{v}\cdot\grad{c}$ in \eqref{eq:CHNS:model:a} in nonconservative form instead of $\div{(c\vec{v})}$ in conservative form. However, for the convenience of proving discrete global mass conservation property, we propose to use the conservative form here. 
\end{remark}
\begin{remark}
The diffusion operator $-\laplace{\vec{v}}$ and the convection operator $\vec{v}\cdot\grad{\vec{v}}$ in \eqref{eq:CHNS:model:c} can be replaced by more general forms of $-2\div{\strain(\vec{v})}$ and $\div{(\vec{v}\otimes\vec{v})}$\,, where the deformation tensor $\strain(\vec{v}) = \frac{1}{2} \big(\grad \vec{v} + \transpose{(\grad\vec{v})}\big)$\,. This equivalence directly comes from the identities
\begin{align*}
2\div{\strain(\vec{v})} &= \grad{(\div{\vec{v}})} + \laplace{\vec{v}},\\
\div(\vec{v}\otimes\vec{v}) &= \vec{v}\cdot\grad{\vec{v}} + (\div{\vec{v}})\vec{v}.
\end{align*}
Note since these operators are mathematically equivalent, one might consider formulating problems more generally. However, the modifications may lead unexpected behaviors of the velocity field at the outflow boundary in open boundary simulation. The details refer to \cite{heywood1996artificial}.
\end{remark}
\begin{remark}
The second term on the right-hand side of \eqref{eq:CHNS:model:c} expresses the phase introduced force. This term appears differently in different literatures and some authors propose the form $-\div(\grad{c}\otimes\grad{c})$~\cite{liu2003phase}, or $\mu\grad{c}$~\cite{Badalassi2003}, or $-c\,\grad\mu$~\cite{DingSpelt2007}. It can be shown that the three expressions are equivalent by redefining the pressure $p$ \cite{feng2006fully}.
\end{remark}

In the rest of this section, we briefly summarize some well-known analytical properties of Cahn--Hilliard--Navier--Stokes model.
\paragraph{Well-posedness}
A weak formulation of the Cahn--Hilliard--Navier--Stokes system \eqref{eq:CHNS:model} is proposed as finding the quaternion $(c,\, \mu,\, \vec{v},\, p)$, where
\begin{align*}
c       \in L^{\infty}\big(0,\, T;\, H^1(\Omega)\big) \cap L^4\big(0,\, T;\, L^\infty(\Omega)\big), &&
\mu     \in L^2\big(0,\, T;\, H^1(\Omega)\big), \\
\vec{v} \in L^2\big(0,\, T;\, H_0^1(\Omega)^d\big) \cap L^{\infty}\big(0,\, T;\, L^2(\Omega)^d\big), &&
p       \in L^2\big(0,\, T;\, L_0^2(\Omega)\big), \\
\partial_t c        \in L^2\big(0,\, T;\, H^{-1}(\Omega)\big), &&
\partial_t \vec{v}  \in L^2\big(0,\, T;\, H^{-1}(\Omega)^d\big),
\end{align*}
such that for a.\,e. $t \in (0,\,T)$,
\begin{subequations}\label{eq:CHNS:weak_formulation}
\begin{eqnarray}
\langle\partial_t c,\, \chi\rangle + \big(\grad{\mu},\, \grad{\chi}\big) - (c\vec{v},\, \grad{\chi}) = 0, \quad \forall \chi \in H^1(\Omega),\\
(\mu,\, \varphi) - \big(\Phi'(c),\, \varphi\big) - \kappa(\grad{c},\, \grad{\varphi}) = 0, \quad \forall \varphi \in H^1(\Omega),\\
\begin{aligned}
\langle\partial_t \vec{v},\, \vec{\theta}\rangle + (\vec{v}\cdot\grad{\vec{v}},\, \vec{\theta}) + \mu_\mathrm{s}(\grad{\vec{v}},\, \grad{\vec{\theta}}) \hspace*{9.5em}\\ - (\div{\vec{\theta}},\, p) + (c\vec{\theta},\, \grad{\mu}) = 0, \quad \forall \vec{\theta} \in H_0^1(\Omega)^d,
\end{aligned}\\
(\div{\vec{v}},\, \phi) = 0, \quad \forall \phi \in L_0^2(\Omega),
\end{eqnarray}
\end{subequations}
with initial data
\begin{align*}
c(0) &\in \big\{c\in H^2(\Omega):~\grad{c}\cdot\normal=0,~\text{on}~\partial\Omega\big\}, \\
\vec{v}(0) &\in \big\{\vec{v}\in H_0^1(\Omega)^d:~(\div{\vec{v}},\, \phi)=0,~\forall\phi\in L_0^2(\Omega)\big\}.
\end{align*}
The $L^2$ inner-product is denoted by $(\cdot,\cdot)$ and the duality pairing by $\langle\cdot,\cdot\rangle$.
Standard notation is used for the Sobolev and Bochner spaces and we recall that $L_0^2(\Omega)$ denotes the space of $L^2$ functions with zero average.
The existence of the weak solution to \eqref{eq:CHNS:weak_formulation} follows the argument as in \cite{diegel2017convergence}. A generalized version of Cahn--Hilliard--Navier--Stokes model, in which the deformation tensor $\strain{(\vec{v})}$ was employed and the capillary stress tensor related term was expressed as $-\div{(\grad{c}\otimes\grad{c})}$\,, was studied in \cite{liu2003phase}. The authors there show that with periodic boundary conditions the existence of weak solution is guaranteed.

\paragraph{Mass conservation}
Let $\bar{c}_0$ denote the mass average at time $t_0$. The solution of the model problem \eqref{eq:CHNS:model} enjoys the global mass conservation property \cite{FLAR2018finite}.
\begin{theorem}\label{thm:CHNS:mass_conservation}
The total amount of the order parameter $c$ is preserved, i.\,e., for any $t\in (0,\, T)$, we have
\begin{equation*}
\frac{1}{\abs{\Omega}} \int_\Omega c = \frac{1}{\abs{\Omega}} \int_\Omega c^0 = \bar{c}_0.
\end{equation*}
\end{theorem}

\paragraph{Energy dissipation}
Benefitting from the boundary conditions (\ref{eq:CHNS:model:out_c}-\ref{eq:CHNS:model:boundary_v})\,, the Cahn--Hilliard--Navier--Stokes model \eqref{eq:CHNS:model} is an energy dissipative system. Analysis of a similar model can be found in \cite{shen2015decoupled}. Define the total energy as follows
\begin{equation}\label{eq:CHNS:energy}
F(c,\vec{v})=\underbrace{\int_\Omega \frac{1}{2}\abs{\vec{v}}^2}_{\text{kinetic energy}} + \underbrace{\int_\Omega \Big(\Phi(c) + \frac{\kappa}{2}\abs{\grad c}^2\Big)}_{\text{Helmholtz free energy}}.
\end{equation}
Then \eqref{eq:CHNS:model} satisfies the following energy dissipation law and for technique details of the proof we refer to \cite{FLSAR2017energy}.
\begin{theorem}
The total energy is non-increasing in time, i.\,e., $\dd_t F(c,\vec{v})(t) \leq 0$ for any $t\in (0,\, T) $. We have the identity
\begin{equation}\label{eq:CHNS:energy_dissipation}
\frac{\dd}{\dd t} \int_\Omega \Big(\frac{1}{2}\abs{\vec{v}}^2 + \Phi(c) + \frac{\kappa}{2}\abs{\grad c}^2\Big) = -\int_{\Omega} \Big(\mu_\mathrm{s}\,\grad{\vec{v}}:\grad{\vec{v}} + \grad{\mu}\cdot\grad\mu\Big) \leq 0.
\end{equation}
\end{theorem}
\begin{remark}
The energy dissipation law still holds in case of changing the diffusion term $-\mu_\mathrm{s}\laplace{\vec{v}}$ in \eqref{eq:CHNS:model:c} to $-2\div\big(\mu_\mathrm{s}\strain(\vec{v})\big)$\,. The modification of the proof can be easily obtained by using the tensor identity 
\begin{equation*}
\div{\big(\mu_\mathrm{s}\strain(\vec{u})\vec{v}\big)} = \mu_\mathrm{s}\,\strain(\vec{u}):\strain(\vec{v}) + \div{\big(\mu_\mathrm{s}\strain(\vec{u})\big)}\cdot\vec{v}.
\end{equation*}
Note here, in this general case, $\mu_\mathrm{s}$ is not required to be a constant any more.
\end{remark}

\paragraph{Chemical energy density}
\begin{subequations}\label{eq:CHNS:chemical_energy_density}
The chemical energy density $\Phi$ may take several forms. Two popular expressions of $\Phi$ are the Ginzburg--Landau double well potential \cite{novick2008cahn},
\begin{equation}\label{eq:CHNS:chemical_energy_density:GL}
\Phi(c) = \frac{1}{4}(1+c)^2(1-c)^2, 
\end{equation}
and the logarithmic potential \cite{barrett1999finite}, 
\begin{equation}\label{eq:CHNS:chemical_energy_density:Log}
\Phi(c) = \frac{\vartheta}{2}\Big((1+c)\log{(\frac{1+c}{2})}+(1-c)\log{(\frac{1-c}{2})}\Big)+\frac{\vartheta_\mathrm{c}}{2}(1-c^2). 
\end{equation}
The parameters $\vartheta$ and $\vartheta_\mathrm{c}$ are positive constants. 
\end{subequations}
%
\paragraph{Convex-concave decomposition}
Throughout our analysis, we assume the chemical energy density $\Phi \in \mathcal{C}^2$, i.\,e., $\Phi$ is a two times continuously differentiable function with respect to $c$. Any $\mathcal{C}^2$ function can be decomposed into the sum of a convex part and a concave part~\cite{YuilleRangarajan2003}. We write 
\begin{equation}\label{eq:convexconcave}
\Phi(c) = \Phi_{+}(c)+\Phi_{-}(c),
\end{equation}
where $\Phi_+$ is a convex function and $\Phi_-$ a concave function. Although the convex-concave splitting for any $\mathcal{C}^2$ function always exists, the decomposition is not unique. We show two examples of splitting for the chemical energy density:
\begin{itemize}
\item[i.] Convex-concave splitting for the Ginzburg--Landau double well potential:
\begin{align*}
\Phi_{+}(c) = \frac{1}{4}(1+c^4), && 
\Phi_{-}(c) = -\frac{1}{2}c^2. 
\end{align*}
\item[ii.] Convex-concave splitting for the logarithmic potential:
\begin{align*}
\Phi_{+}(c) = \frac{\vartheta}{2}\Big((1+c)\log{(\frac{1+c}{2})}+(1-c)\log{(\frac{1-c}{2})}\Big), &&
\Phi_{-}(c) = \frac{\vartheta_\mathrm{c}}{2}(1-c^2).
\end{align*}
\end{itemize}
\begin{remark}
The logarithmic potential is also called Flory--Huggins potential. For some choices of the parameters $\vartheta$ and $\vartheta_\mathrm{c}$\,, the minimum of this potential may take negative value. Although the logarithmic potential is only defined on a finite interval, many authors define an extension over $\IR$ for convenience in numerical simulations, for instance see \cite{barrett1999finite,yang2017linear}.
\end{remark}

\section{Numerical Analysis}\label{sec:CHNS:numerical_analysis}
In this section, we introduce an interior penalty discontinuous Galerkin method for the Cahn--Hilliard--Navier--Stokes system and analyze their numerical properties. These include uniquely solvability of the scheme, discrete mass conservation, energy dissipation, stability and error bounds. Our results are valid
for any general form for the chemical energy density at the exception of the priori error bounds that are valid with additional assumptions on the chemical energy density. 

\subsection{Preliminaries}\label{sec:CHNS:preliminaries}
\paragraph{Domain and triangulation}
Let $\setE_h = \{E_k\}$ be a family of conforming nondegenerate (also called regular) meshes of the domain $\Omega$. The parameter $h$ denotes the maximum element diameter. Let $\Gammah$ denote the set of interior faces. For each interior face $e \in \Gammah$ shared by elements $E_{k^-}$ and $E_{k^+}$, we define a unit normal vector $\normal_e$ that points from $E_{k^-}$ into $E_{k^+}$. For the face $e$ on boundary $\partial\Omega$, i.\,e., $e = E_{k^-} \cap \partial\Omega$, the normal $\normal_e$ is taken to be the unit outward vector to $\partial\Omega$. We also denote by $\normal_E$ the unit normal vector outward to the element $E$. The natural spaces to work with DG methods are the broken Sobolev spaces. For any real number $r$, we introduce 
\begin{equation*}
H^r(\setE_h) = \big\{\omega\in L^2(\Omega):~\forall E \in \setE_h,\, \on{\omega}{E} \in H^r(E) \big\}. 
\end{equation*}
The average and jump of any scalar quantity $\omega$ is defined for each interior face $e\in\Gammah$ by
\begin{align*}
\avg{\omega} = \frac{1}{2}\on{\omega}{E_{k^-}}\! + \frac{1}{2}\on{\omega}{E_{k^+}}, &&
\jump{\omega} = \on{\omega}{E_{k^-}}\! - \on{\omega}{E_{k^+}}.
\end{align*}
If $e$ belongs to the boundary $\partial\Omega$, the jump and average of $\omega$ coincide with its trace on $e$. The related definitions of any vector quantity in $H^r(\setE_h)^d$ are similar~\cite{riviere2008}.

\paragraph{DG forms}
We introduce the forms
\begin{align*}
a_{\mathcal{A}}: H^2(\setE_h) \times H^2(\setE_h)^d \times H^2(\setE_h) &\rightarrow \IR, \\ a_{\mathcal{C}}: H^2(\setE_h)^d \times H^2(\setE_h)^d \times H^2(\setE_h)^d \times H^2(\setE_h)^d &\rightarrow \IR, \\
a_{\mathcal{D}}: H^2(\setE_h) \times H^2(\setE_h) &\rightarrow \IR, \\
a_{\strain}: H^2(\setE_h)^d \times H^2(\setE_h)^d &\rightarrow \IR, \\
b_\mathcal{P}: H^1(\setE_h) \times H^2(\setE_h)^d &\rightarrow \IR, \\
b_\mathcal{I}: H^2(\setE_h) \times H^2(\setE_h) \times H^2(\setE_h)^d &\rightarrow \IR,
\end{align*} 
corresponding to DG discretization of the advection term $\div{(c\vec{v})}$\,, convection term $\vec{v}\cdot\grad{\vec{v}}$\,, elliptic operator $-\laplace{c}$\,, diffusion term $-\laplace{\vec{v}}$\,, pressure term $\grad{p}$\,, and interface term $-c\grad{\mu}$\,, respectively
\begin{subequations}\label{eq:CHNS:DG_forms}
\begin{align}
a_{\mathcal{A}}(c,\vec{v},\chi) =&  
-\sum_{E\in\setE_h} \int_E c\,\vec{v}\cdot\grad{\chi}
+\sum_{e\in\Gammah} \int_e \avg{c} \avg{\vec{v}\cdot\normal_e} \jump{\chi},
\label{eq:CHNS:DG_advection} \\
\begin{split}\label{eq:CHNS:DG_convection}
a_{\mathcal{C}}(\vec{w},\vec{v},\vec{z},\vec{\theta}) =&
\sum_{E\in\setE_h}\Big(\int_E (\vec{v}\cdot\grad\vec{z})\cdot\vec{\theta} 
+ \int_{\partial E_{-}^\vec{w}} \abs{\avg{\vec{v}}\cdot\normal_E}\,(\vec{z}^\mathrm{int}-\vec{z}^\mathrm{ext})\cdot\vec{\theta}^\mathrm{int}\Big)\\
&+ \frac{1}{2} \sum_{E\in\setE_h}\int_E (\div{\vec{v}})\,\vec{z}\cdot{\vec{\theta}}
- \frac{1}{2} \sum_{e\in\Gammah\cup\partial\Omega}\int_e \jump{\vec{v}\cdot\normal_e}\avg{\vec{z}\cdot\vec{\theta}},
\end{split}\\
\begin{split}\label{eq:CHNS:DG_diffusion}
a_{\mathcal{D}}(c,\chi) =&
\sum_{E\in\setE_h} \int_E \grad c \cdot \grad \chi
-\sum_{e\in\Gammah} \int_e \avg{\grad c \cdot \normal_e} \jump{\chi}\\
&-\sum_{e\in\Gammah} \int_e \avg{\grad \chi \cdot \normal_e} \jump{c}
+ \frac{\sigma}{h} \sum_{e\in\Gammah}\int_e \jump{c}\jump{\chi},
\end{split}\\
\begin{split}\label{eq:CHNS:DG_strain}
a_\strain(\vec{v}, \vec{\theta}) =& 
\sum_{E\in\setE_h} \int_E \grad{\vec{v}}:\grad{\vec{\theta}} 
- \sum_{e\in\Gammah\cup\partial\Omega} \int_e\avg{\grad{\vec{v}}\cdot\normal_e}\cdot\jump{\vec{\theta}}\\
&-\sum_{e\in\Gammah\cup\partial\Omega} \int_e\avg{\grad{\vec{\theta}}\cdot\normal_e}\cdot\jump{\vec{v}}
+ \frac{\sigma}{h}\sum_{e\in\Gammah\cup\partial\Omega} \int_e \jump{\vec{v}}\cdot\jump{\vec{\theta}},
\end{split}\\
b_{\mathcal{P}}(p,\vec{\theta}) =&
-\sum_{E\in\setE_h} \int_E p\div{\vec{\theta}}
+\sum_{e\in\Gammah\cup\partial\Omega} \int_e \avg{p}\jump{\vec{\theta}\cdot\normal_e},
\label{eq:CHNS:DG_pressure}\\
b_\mathcal{I}(c,\mu,\vec{\theta}) =& 
-\sum_{E\in\setE_h} \int_E c\,\grad{\mu}\cdot\vec{\theta}\,
+\sum_{e\in\Gammah} \int_e \avg{c}\jump{\mu}\avg{\vec{\theta}\cdot\normal_e}.
\label{eq:CHNS:DG_interface}
\end{align}
\end{subequations}
In \eqref{eq:CHNS:DG_convection}, the set $\partial E_{-}^\vec{w} = \big\{\vec{x}\in\partial E: \avg{\vec{w}}\cdot\normal_E <0 \big\}$ and the superscript $\mathrm{int}$ (resp. $\mathrm{ext}$) refers to the trace of the function on a face of $E$ coming from the interior of $E$ (resp. coming from the exterior of $E$ on that face), in addition, if the face lies on the boundary of the domain, we take the exterior trace to be zero. 
For more details related to \eqref{eq:CHNS:DG_convection}, we refer the reader to \cite{girault2005discontinuous}. 
The derivation of these DG forms are given in \cite{riviere2008}. 
We recall that since we use symmetric bilinear forms the penalty parameter, $\sigma$, has to be chosen large enough.
\begin{remark}\label{rem:CHNS:relation_aA_bI}
In order to ensure the discrete unconditional energy dissipation in a convenient manner, here we specially design our numerical discretization to satisfy $a_{\mathcal{A}}(c,\vec{v},\mu)=b_\mathcal{I}(c,\mu,\vec{v})$ for any $c, \mu$ and  $\vec{v}$. 
\end{remark}

\subsection{Numerical scheme}\label{sec:CHNS:DGscheme}
\paragraph{DG scheme}
Uniformly partition $[0,\,T]$ into $N$ subintervals and let $\tau$ be the time step length. For any fixed positive integer $q\in\IN_+$\,, the set $\IP_q(E)$ denotes all polynomials of degree at most $q$ on an element $E$. Define the following broken polynomial spaces
\begin{align*}
S_h &= \big\{\omega\in L^2(\Omega):~\forall E \in \setE_h,\, \on{\omega}{E} \in \IP_q(E) \big\}, \quad M_h = S_h \cap L_0^2(\Omega),\\
\mathbf{X}_h &= \big\{\vec{\theta} \in L^2(\Omega)^d:~\forall E \in \setE_h,\, \on{\vec{\theta}}{E} \in \IP_{q}(E)^d \big\},\\
Q_h &= \big\{\omega\in L^2_0(\Omega):~\forall E \in \setE_h,\, \on{\omega}{E} \in \IP_{q-1}(E) \big\}, \\
\mathbf{V}_h &= \big\{\vec{\theta} \in \mathbf{X}_h:~\forall \phi \in Q_h, ~b_{\mathcal{P}}(\phi,\vec{\theta}) = 0\big\},
\end{align*}
We employ the implicit Euler method with Picard's linearization for temporal discretization. The fully discrete mixed convex-concave splitting DG scheme reads: \\
for any $1\leq n \leq N$\,, given $c_h^{n-1} \in S_h$ and $\vec{v}_h^{n-1} \in \mathbf{X}_h$ find $(c_h^n, \mu_h^n, \vec{v}_h^{n}, p_h^n) \in S_h \times S_h \times \mathbf{X}_h \times Q_h$ such that
\begin{subequations}\label{eq:CHNS:DGscheme}
\begin{eqnarray}
(\delta_\tau c_h^n,\chi) + a_{\mathcal{D}}(\mu_h^n,\chi) + a_{\mathcal{A}}(c_h^{n-1},\vec{v}_h^n,\chi) = 0, \quad \forall \chi \in S_h, \label{eq:CHNS:DGscheme:a}\\
(\Phi_{+}\,\!'(c_h^n)+\Phi_{-}\,\!'(c_h^{n-1}),\varphi) + \kappa a_{\mathcal{D}}(c_h^n,\varphi) - (\mu_h^n,\varphi) = 0, \quad \forall \varphi \in S_h, \label{eq:CHNS:DGscheme:b}\\
\begin{aligned}
(\delta_\tau \vec{v}_h^n ,\vec{\theta}) + a_{\mathcal{C}}(\vec{v}_h^{n-1},\vec{v}_h^{n-1},\vec{v}_h^n,\vec{\theta}) + \mu_\mathrm{s} a_\strain(\vec{v}_h^n, \vec{\theta})\hspace*{7.875em}\\ + b_{\mathcal{P}}(p_h^n,\vec{\theta}) - b_\mathcal{I}(c_h^{n-1},\mu_h^n,\vec{\theta}) = 0, \quad \forall \vec{\theta} \in \mathbf{X}_h,
\end{aligned}\label{eq:CHNS:DGscheme:c}\\
b_{\mathcal{P}}(\phi,\vec{v}_h^n) = 0, \quad \forall \phi \in Q_h. \label{eq:CHNS:DGscheme:d}
\end{eqnarray}
\end{subequations}
Here, $\delta_\tau$ denotes the backward finite temporal difference operator:
\[
\delta_\tau c_h^n = \frac{c_h^n-c_h^{n-1}}{\tau}.
\]
The initial data $c_h^0$ and $\vec{v}_h^0$ are good approximations of $c^0$ and $\vec{v}^0$ respectively.
For instance, we choose $\vec{v}_h^0$ as the $L^2$ projection of $\vec{v}^0$ into $\mathbf{X}_h$ and we 
choose $c_h^0 = \mathcal{P}_h c^0$\,, where $\mathcal{P}_h: H^2(\setE_h) \rightarrow S_h$ is the elliptic projection operator:
\begin{align}\label{eq:CHNS:elliptic_proj_c}
a_{\mathcal{D}}(\mathcal{P}_h c - c,\chi) = 0, && \forall \chi \in S_h, && \text{with constraint} && (\mathcal{P}_h c - c,1) = 0,
\end{align}

\paragraph{Operator properties}
Throughout this paper, the semi-norms $\norm{\cdot}{\mathrm{DG}}$ for any scalar quantity $c \in H^1(\setE_h)$ and for any vector quantity $\vec{v} \in H^1(\setE_h)^d$ are defined as follows, respectively
\begin{align*}
\forall c &\in H^1(\setE_h), & \norm{c}{\mathrm{DG}}^2 &= \sum_{E\in\setE_h} \norm{\grad{c}}{L^2(E)}^2 + \frac{\sigma}{h}\sum_{e\in\Gammah} \norm{\jump{c}}{L^2(e)}^2,\\
\forall \vec{v} &\in H^1(\setE_h)^d, & \norm{\vec{v}}{\mathrm{DG}}^2 &= \sum_{E\in\setE_h} \norm{\grad{\vec{v}}}{L^2(E)}^2 + \frac{\sigma}{h}\sum_{e\in\Gammah\cup\partial\Omega} \norm{\jump{\vec{v}}}{L^2(e)}^2.
\end{align*}
Note $\norm{\cdot}{\mathrm{DG}}$ is a norm on $H^1(\setE_h) \cap L_0^2(\Omega)$ and due to the fact that the $d-1$ dimensional Lebesgue measure of $\partial\Omega$ is positive, $\norm{\cdot}{\mathrm{DG}}$ is a norm on $H^1(\setE_h)^d$ as well. Furthermore, the spaces $H^1(\setE_h) \cap L_0^2(\Omega)$ and $H^1(\setE_h)^d$ equipped with above energy norm $\norm{\cdot}{\mathrm{DG}}$ are reflexive Hilbert spaces. 
For readibility, we drop the dimension in the norm notation for vector functions.
Let $p_0$ be the exponent of the Sobolev embedding of $H^1(\Omega)$ into $L^p(\Omega)$ defined by $\frac{1}{p_0} = \frac{1}{2} - \frac{1}{d}$. Then, we have the following result
\begin{lemma}[Poincar\'e's inequality \cite{girault2017strong}]
For each $p \leq p_0$ (exclude infinity when $d=2$), there exists a constant $C_P > 0$ independent of mesh size $h$ such that
\begin{equation*}
\norm{\chi - \frac{1}{\abs{\Omega}}\int_{\Omega}\chi}{L^p(\Omega)} \leq C_P\norm{\chi}{\mathrm{DG}}, \quad\forall \chi\in S_h.
\end{equation*}
We also have 
\[
\norm{\vec{\theta}}{L^p(\Omega)} \leq C_P \norm{\vec{\theta}}{\mathrm{DG}},\quad\forall \vec{\theta}\in \mathbf{X}_h.
\]
\end{lemma}
Many of the DG forms above satisfy important properties for the analysis of our scheme. Below,  we recall several well-known results and provide a briefly proof for the boundedness of $a_{\mathcal{A}}$. We omit the other proofs for the sake of brevity -- for details see \cite{ChaabaneGiraultPuelzRiviere2017,girault2005splitting,riviere2008}.
\begin{lemma}[Boundedness of $a_{\mathcal{A}}$]\label{lem:CHNS:boundedness_aA}
There exists a constant $C_\gamma>0$ independent of mesh size $h$ such that for all $c,\chi$ in $S_h$ and $\vec{v}$ in $\mathbf{X}_h$, the following bounds hold:
\begin{subequations}
\begin{align}
\abs{a_{\mathcal{A}}(c,\vec{v},\chi)} \leq& 
C_\gamma \Big(\norm{c}{\mathrm{DG}} + \abs{\int_{\Omega}c}\Big)\norm{\vec{v}}{\mathrm{DG}}\norm{\chi}{\mathrm{DG}},\label{eq:CHNS:boundedness_aA_1}\\
\abs{a_{\mathcal{A}}(c,\vec{v},\chi)} \leq& 
C_\gamma \Big(\norm{c}{\mathrm{DG}} + \abs{\int_{\Omega}c}\Big)\norm{\vec{v}}{L^2(\Omega)}^{1/2}\norm{\vec{v}}{\mathrm{DG}}^{1/2}\norm{\chi}{\mathrm{DG}}.\label{eq:CHNS:boundedness_aA_2}
\end{align}
\end{subequations}
In particular, for all $c$ in $M_h$, $\chi$ in $S_h$, and $\vec{v}$ in $\mathbf{X}_h$, the first inequality above implies $\abs{a_{\mathcal{A}}(c,\vec{v},\chi)} \leq C_\gamma\norm{c}{\mathrm{DG}}\norm{\vec{v}}{\mathrm{DG}}\norm{\chi}{\mathrm{DG}}$\,.
\end{lemma}
\begin{proof}
The inequalities above are direct results of suitably bounding \cref{eq:CHNS:DG_advection} term after term. Using H\"older's inequality and Cauchy--Schwarz's inequality we have
\begin{align*}
\abs{\sum_{E\in\setE_h} \int_E c\,\vec{v}\cdot\grad{\chi}} 
\leq \Big(\sum_{E\in\setE_h} \norm{c}{L^4(E)}^4\Big)^{\frac{1}{4}}\Big(\sum_{E\in\setE_h} \norm{\vec{v}}{L^4(E)}^4\Big)^{\frac{1}{4}}\Big(\sum_{E\in\setE_h} \norm{\grad{\chi}}{L^2(E)}^2\Big)^{\frac{1}{2}}.
\end{align*}
Again, using H\"older's inequality, Cauchy--Schwarz's inequality, triangular inequality and trace inequality we get
\begin{align*}
&\abs{\sum_{e\in\Gammah} \int_e \avg{c} \avg{\vec{v}\cdot\normal_e} \jump{\chi}}\\
\leq& \Big(\sum_{e\in\Gammah} \norm{\avg{c}}{L^4(e)}^4\Big)^\frac{1}{4} \Big(\sum_{e\in\Gammah}\norm{\avg{\vec{v}\cdot\normal_e}}{L^4(e)}\Big)^\frac{1}{4} \Big(\sum_{e\in\Gammah} \norm{\jump{\chi}}{L^2(e)}^2\Big)^\frac{1}{2}\\
\leq& C\Big(\sum_{E\in\setE_h} \norm{c}{L^4(E)}^4\Big)^\frac{1}{4} \Big(\sum_{E\in\setE_h}\norm{\vec{v}}{L^4(E)}\Big)^\frac{1}{4} \Big(\frac{1}{h}\sum_{e\in\Gammah} \norm{\jump{\chi}}{L^2(e)}^2\Big)^\frac{1}{2}.
\end{align*}
Thus, combine these bounds, by definition of $a_{\mathcal{A}}$ and Poincar\'e's inequality we obtain \cref{eq:CHNS:boundedness_aA_1}. For the inequality \cref{eq:CHNS:boundedness_aA_2}, using similar arguments as above, we have
\begin{equation*}
\abs{a_{\mathcal{A}}(c,\vec{v},\chi)} \leq C \norm{c}{L^6(\Omega)}\norm{\vec{v}}{L^3(\Omega)}\norm{\chi}{\mathrm{DG}}.
\end{equation*}
Finally, we conclude our proof by applying Poincar\'e's inequality and interpolation inequality $\norm{\vec{v}}{L^3(\Omega)}^2 \leq \norm{\vec{v}}{L^2(\Omega)}\norm{\vec{v}}{L^6(\Omega)}$.
\end{proof}
\begin{lemma}[Continuity of $a_{\mathcal{C}}$]\label{lem:CHNS:continuityaC}
The form $a_\mathcal{C}$ is linear with respect to its second to fourth arguments and there exists a constant $C_\nu>0$ independent of mesh size $h$ such that for all $\vec{u}_h, \vec{v}_h, \vec{w}_h, \vec{z}_h$ in $\mathbf{X}_h + (H_0^1(\Omega))^d$,
\[
\abs{a_{\mathcal{C}}(\vec{z}_h,\vec{u}_h,\vec{v}_h,\vec{w}_h)} 
\leq C_\nu\norm{\vec{u}_h}{\mathrm{DG}}\norm{\vec{v}_h}{\mathrm{DG}}\norm{\vec{w}_h}{\mathrm{DG}}.
\]
\end{lemma}
\begin{lemma}[Bounds of $a_{\mathcal{C}}$]\label{lem:CHNS:boundconvection}
There exists a constant $C$ independent of $h$ such that for any $\vec{u}$ in $(L^\infty(\Omega)\cap W^{1,3}(\Omega))^d$,
any $\vec{v}_h$ in $\mathbf{V}_h$ and any $\vec{w}_h, \vec{z}_h$ in $\vec{X}_h$, the bound holds
\[
|a_\mathcal{C}(\vec{z}_h,\vec{v}_h,\vec{u},\vec{w}_h)| \leq C \left( \Vert \vec{u}\Vert_{L^\infty(\Omega)}
+ \vert \vec{u}\vert_{W^{1,3}(\Omega)}\right) \Vert \vec{v}_h\Vert_{L^2(\Omega)} \Vert \vec{w}_h\Vert_{\mathrm{DG}}.
\]
\end{lemma}
\begin{lemma}[Positivity of $a_{\mathcal{C}}$]\label{lem:CHNS:positivity_convection}
The form $a_\mathcal{C}$ satisfies the positivity property, i.\,e., for all $\vec{v},\vec{z}$ in $\mathbf{X}_h$\,,
\begin{equation*}
a_\mathcal{C}(\vec{v},\vec{v},\vec{z},\vec{z}) 
= \frac{1}{2} \sum_{E\in\setE_h}\int_{\partial E_{-}^\vec{v}} \abs{\avg{\vec{v}}\cdot\normal_E}\,\Lnorm{\vec{z}^\mathrm{ext}-\vec{z}^\mathrm{int}}^2 
\geq 0.
\end{equation*}
\end{lemma}
\begin{lemma}[Continuity of $a_{\mathcal{D}}$]
\!The bilinear form $a_{\mathcal{D}}$ is continuous on $S_h$ equipped with the energy norm, i.\,e., there exists a constant $C_\alpha>0$ independent of mesh size $h$ such that for all $c,\chi$ in $S_h$\,,~$\abs{a_{\mathcal{D}}(c,\chi)} \leq C_\alpha\norm{c}{\mathrm{DG}}\norm{\chi}{\mathrm{DG}}$\,.
\end{lemma} 
\begin{lemma}[Coercivity of $a_{\mathcal{D}}$]
Assume that $\sigma$ is sufficiently large. Then, there exists a constant $K_\alpha>0$ independent of mesh size $h$ such that 
\[
a_{\mathcal{D}}(c,c) \geq K_\alpha\norm{c}{\mathrm{DG}}^2, \quad \forall c \in S_h.
\]
\end{lemma}
\begin{lemma}[Continuity of $a_\strain$]
The bilinear form $a_\strain$ is continuous on $\mathbf{X}_h$ equipped with the energy norm, i.\,e., there exists a constant $C_\strain>0$ independent of mesh size $h$ such that for all $\vec{v},\vec{\theta}$ in $\mathbf{X}_h$\,,~$\abs{a_\strain(\vec{v},\vec{\theta})} \leq C_\strain\norm{\vec{v}}{\mathrm{DG}}\norm{\vec{\theta}}{\mathrm{DG}}$\,.
\end{lemma} 
\begin{lemma}[Coercivity of $a_\strain$]\label{lem:astraincoer}
Assume that $\sigma$ is sufficiently large. Then, there exists a constant $K_\strain>0$ independent of mesh size $h$ such that
\[
a_\strain(\vec{v},\vec{v}) \geq K_\strain\norm{\vec{v}}{\mathrm{DG}}^2, \quad \forall \vec{v}\in\mathbf{X}_h.
\]
\end{lemma}
\begin{lemma}[Inf-sup]\label{lem:infsup}
There exists a constant $\beta > 0$, independent of mesh size $h$, such that
\begin{equation*}
\inf_{\phi \in Q_h}\,\sup_{\vec{\theta} \in \mathbf{X}_h} \frac{b_{\mathcal{P}}(\phi,\vec{\theta})}{\norm{\phi}{L^2(\Omega)}\norm{\vec{\theta}}{\mathrm{DG}}} \geq
\beta.
\end{equation*}
\end{lemma}


\subsection{Discrete mass conservation}\label{sec:CHNS:DG_mass_conservation}
\begin{theorem}\label{thm:CHNS:discrete_mass_conservation}
The DG scheme \eqref{eq:CHNS:DGscheme} satisfies the discrete global mass conservation property, i.\,e., for any $1\leq n \leq N$\,, we have
\begin{equation*}
(c_h^n,1) = (c_h^{0},1) = (c^0,1) = \big(c(t^n),1\big).
\end{equation*}
\end{theorem}
\begin{proof}
The proof for the first equality is straightforward and obtained by choosing $\chi = 1$ in \eqref{eq:CHNS:DGscheme:a} and by using $a_{\mathcal{D}}(\mu_h^n,1)=0$ and $a_{\mathcal{A}}(c_h^n,\vec{v}_h^n,1)=0$. Furthermore, applying \eqref{eq:CHNS:elliptic_proj_c} and \cref{thm:CHNS:mass_conservation}, we obtain the second and third equalities.
\end{proof}
\begin{remark}
One interesting property that naturally comes with the primal DG scheme is the conservation of mass on each mesh element. For instance, we fix an element $E$ that belongs to the interior of the domain, i.\,e., $\partial E \cap \partial \Omega = \emptyset$. It is easy to show the PDE solution of the Cahn--Hilliard--Navier--Stokes system satisfies
\begin{equation*}
\frac{\dd}{\dd t}\int_E c - \int_{\partial E} \grad{\mu}\cdot\normal_E + \int_{\partial E} c\vec{v}\cdot\normal_E = 0.
\end{equation*}
The artificial numerical mass can be exactly computed by choosing $\chi = 1$ on $E$ and vanishes elsewhere in \eqref{eq:CHNS:DGscheme:a}. Then we obtain a balance equation
\begin{multline*}
\frac{1}{\tau}\int_E (c_h^n-c_h^{n-1})\, 
\underbrace{- \int_{\partial E} \avg{\grad{\mu_h^n}}\cdot\normal_E + \int_{\partial E} \avg{c_h^n}\avg{\vec{v}_h^n}\cdot\normal_E}_{\text{flux of mass passing through}~\partial E}\\
= \underbrace{-\frac{\sigma}{h}\int_{\partial E} \Big((\mu_h^n)^\mathrm{int}-(\mu_h^n)^\mathrm{ext}\Big)}_{\text{artificial mass due to the penalty}}.
\end{multline*} 
\end{remark}

\subsection{Existence and uniqueness}\label{sec:CHNS:DGexistence_uniqueness}
Investigating the unique solvability of the fully discrete DG method \eqref{eq:CHNS:DGscheme} is a complicated task. We will design an equivalent scheme, which is based on an auxiliary flow problem, to overcome this challenge. The existence and uniqueness of the solution for our equivalent scheme can be proved by using nonlinear operator analysis techniques. To this end, we begin our argument by introducing the following auxiliary flow problem: for any $1\leq n \leq N$, given $\vec{v}_h^{n-1} \in \mathbf{X}_h$ find $(\tilde{\vec{v}}_h^n,\tilde{p}_h^n) \in \mathbf{X}_h \times Q_h$ such that
\begin{subequations}\label{eq:CHNS:solvability:aux}
\begin{eqnarray}
\begin{aligned}
\frac{1}{\tau}(\tilde{\vec{v}}_h^n-\vec{v}_h^{n-1},\vec{\theta}) + a_{\mathcal{C}}(\vec{v}_h^{n-1},\vec{v}_h^{n-1},\tilde{\vec{v}}_h^n,\vec{\theta}) \hspace*{12em}\\+ \mu_\mathrm{s} a_\strain(\tilde{\vec{v}}_h^n, \vec{\theta}) + b_{\mathcal{P}}(\tilde{p}_h^n,\vec{\theta}) = 0, \quad \forall \vec{\theta} \in \mathbf{X}_h,
\end{aligned}\label{eq:CHNS:solvability:aux:a}\\
b_{\mathcal{P}}(\phi,\tilde{\vec{v}}_h^n) = 0, \quad \forall \phi \in Q_h.\label{eq:CHNS:solvability:aux:b}
\end{eqnarray}
\end{subequations}
\begin{lemma}\label{lem:CHNS:unique_sol_aux_flow}
There exists a unique solution to the auxiliary flow problem \eqref{eq:CHNS:solvability:aux} for any mesh size $h$ and time step size $\tau$. 
\end{lemma}
\begin{proof}
We first show existence and uniqueness of $\tilde{\vec{v}}_h^n\in\mathbf{V}_h$ satisfying
\begin{align*}
\frac{1}{\tau}(\tilde{\vec{v}}_h^n-\vec{v}_h^{n-1},\vec{\theta}) 
+ a_{\mathcal{C}}(\vec{v}_h^{n-1},\vec{v}_h^{n-1},\tilde{\vec{v}}_h^n,\vec{\theta}) 
+ \mu_\mathrm{s} a_\strain(\tilde{\vec{v}}_h^n, \vec{\theta}) = 0, && \forall \vec{\theta} \in \mathbf{V}_h.
\end{align*}
Since the problem is linear and finite-dimensional, it suffices to show uniqueness. This is easily obtained by using positivity of $a_\mathcal{C}$ and coercivity of $a_\strain$ (see \cref{lem:CHNS:positivity_convection} and \cref{lem:astraincoer}). To recover the discrete pressure $\tilde{p}_h^n\in Q_h$, we then use the inf-sup condition of \cref{lem:infsup}.
\end{proof}

Owing to the last result, we can construct the following scheme by employing the unique discrete solution from the auxiliary flow problem: for any $1\leq n \leq N$\,, given $(y_h^{n-1},\vec{v}_h^{n-1}) \in M_h\times\mathbf{X}_h$, and corresponding $(\tilde{\vec{v}}_h^n,\tilde{p}_h^n)$ satisfying \eqref{eq:CHNS:solvability:aux}, find $(y_h^n,w_h^n,\hat{\vec{v}}_h^n,\hat{p}_h^n) \in M_h \times M_h \times \mathbf{X}_h \times Q_h$ such that
\begin{subequations}\label{CHNS:solvability:schemeB}
\begin{eqnarray}
\frac{1}{\tau}(y_h^n-\hat{y}_h^{n-1},\mathring{\chi}) + a_{\mathcal{D}}(w_h^n,\mathring{\chi}) + a_{\mathcal{A}}(y_h^{n-1}+\bar{c}_0,\hat{\vec{v}}_h^n,\mathring{\chi}) = 0, \quad \forall \mathring{\chi} \in M_h,\label{CHNS:solvability:schemeB:a}\\
\hspace*{2.5em} 
\big(\Phi_{+}\,\!'(y_h^n+\bar{c}_0)+\Phi_{-}\,\!'(y_h^{n-1}+\bar{c}_0),\mathring{\varphi}\big) + \kappa a_{\mathcal{D}}(y_h^n,\mathring{\varphi}) - (w_h^n,\mathring{\varphi}) = 0, \quad \forall \mathring{\varphi} \in M_h,\label{CHNS:solvability:schemeB:b}\\
\begin{aligned}
\frac{1}{\tau}(\hat{\vec{v}}_h^n ,\vec{\theta}) + a_{\mathcal{C}}(\vec{v}_h^{n-1},\vec{v}_h^{n-1},\hat{\vec{v}}_h^n,\vec{\theta}) + \mu_\mathrm{s} a_\strain(\hat{\vec{v}}_h^n, \vec{\theta}) \hspace*{13.2em}\\ + b_{\mathcal{P}}(\hat{p}_h^n,\vec{\theta}) - b_\mathcal{I}(y_h^{n-1}+\bar{c}_0,w_h^n,\vec{\theta}) = 0, \quad \forall \vec{\theta} \in \mathbf{X}_h,
\end{aligned}\label{CHNS:solvability:schemeB:c}\\
b_{\mathcal{P}}(\phi,\hat{\vec{v}}_h^n) = 0, \quad \forall \phi \in Q_h,\label{CHNS:solvability:schemeB:d}
\end{eqnarray}
\end{subequations}
where the initial datum is  defined to be $y_h^0 = \mathcal{P}_h c^0-\bar{c}_0$ and  we recall the
initial velocity $\vec{v}_h^0$ is the L$^2$ projection of $\vec{v}^0$ onto $\mathbf{X}_h$. We also denote
 $\hat{y}_h^{n-1} \in M_h$ the solution of
\begin{align*}
(y_h^{n-1},\mathring{\chi}) - \tau a_{\mathcal{A}}(y_h^{n-1}+\bar{c}_0,\tilde{\vec{v}}_h^n,\mathring{\chi}) = (\hat{y}_h^{n-1},\mathring{\chi}), && \forall \mathring{\chi} \in M_h,
\end{align*}
whose existence and uniqueness are asserted by the Riesz representation theorem. Our next goal is to prove the scheme \eqref{CHNS:solvability:schemeB} is equivalent to the DG scheme \eqref{eq:CHNS:DGscheme}. Due to the translational invariance of the trilinear form $a_{\mathcal{A}}$ with respect to the third argument and using the same techniques as in \cite{LFR2018numerical}, we have
\begin{lemma}\label{lem:CHNS:solvability:P}
The unique solvability of the DG scheme \eqref{eq:CHNS:DGscheme} is equivalent to the unique solvability of the problem: for any $1\leq n \leq N$\,, given $(y_h^{n-1},\vec{v}_h^{n-1}) \in M_h\times\mathbf{X}_h$ find $(y_h^n,w_h^n,\vec{v}_h^n,p_h^n) \in M_h \times M_h \times \mathbf{X}_h \times Q_h$ such that
\begin{subequations}\label{eq:CHNS:solvability:P}
\begin{eqnarray}
(\delta_\tau y_h^n,\mathring{\chi}) + a_{\mathcal{D}}(w_h^n,\mathring{\chi}) + a_{\mathcal{A}}(y_h^{n-1}+\bar{c}_0,\vec{v}_h^n,\mathring{\chi}) = 0, \quad \forall \mathring{\chi} \in M_h,\\
\hspace*{2.5em}
\big(\Phi_{+}\,\!'(y_h^n+\bar{c}_0)+\Phi_{-}\,\!'(y_h^{n-1}+\bar{c}_0),\mathring{\varphi}\big) + \kappa a_{\mathcal{D}}(y_h^n,\mathring{\varphi}) - (w_h^n,\mathring{\varphi}) = 0, \quad \forall \mathring{\varphi} \in M_h,\\
\begin{aligned}
(\delta_\tau \vec{v}_h^n ,\vec{\theta}) + a_{\mathcal{C}}(\vec{v}_h^{n-1},\vec{v}_h^{n-1},\vec{v}_h^n,\vec{\theta}) + \mu_\mathrm{s} a_\strain(\vec{v}_h^n, \vec{\theta}) \hspace*{12.7em}\\ + b_{\mathcal{P}}(p_h^n,\vec{\theta}) - b_\mathcal{I}(y_h^{n-1}+\bar{c}_0,w_h^n,\vec{\theta}) = 0, \quad \forall \vec{\theta} \in \mathbf{X}_h,
\end{aligned}\\
b_{\mathcal{P}}(\phi,\vec{v}_h^n) = 0, \quad \forall \phi \in Q_h.
\end{eqnarray}
\end{subequations}
\end{lemma}
\begin{proof}
It is easy to check the unique solvability of DG scheme \eqref{eq:CHNS:DGscheme} is equivalent to the unique solvability of the problem: for any $1\leq n \leq N$\,, given $y_h^{n-1} \in M_h$ and $\vec{v}_h^{n-1} \in \mathbf{X}_h$ find $(y_h^n, \mu_h^n, \vec{v}_h^{n}, p_h^n) \in M_h \times S_h \times \mathbf{X}_h \times Q_h$ such that
\begin{subequations}\label{eq:CHNS:solvability:P_aux}
\begin{eqnarray}
(\delta_\tau y_h^n,\mathring{\chi}) + a_{\mathcal{D}}(\mu_h^n,\mathring{\chi}) + a_{\mathcal{A}}(y_h^{n-1}+\bar{c}_0,\vec{v}_h^n,\mathring{\chi}) = 0, \quad \forall \mathring{\chi} \in M_h,\\
\hspace*{2.5em}
(\Phi_{+}\,\!'(y_h^n+\bar{c}_0)+\Phi_{-}\,\!'(y_h^{n-1}+\bar{c}_0),\varphi) + \kappa a_{\mathcal{D}}(y_h^n,\varphi) - (\mu_h^n,\varphi) = 0, \quad \forall \varphi \in S_h,\\
\begin{aligned}
(\delta_\tau \vec{v}_h^n ,\vec{\theta}) + a_{\mathcal{C}}(\vec{v}_h^{n-1},\vec{v}_h^{n-1},\vec{v}_h^n,\vec{\theta}) + \mu_\mathrm{s} a_\strain(\vec{v}_h^n, \vec{\theta})\hspace*{7.875em}\\ + b_{\mathcal{P}}(p_h^n,\vec{\theta}) - b_\mathcal{I}(y_h^{n-1}+\bar{c}_0,\mu_h^n,\vec{\theta}) = 0, \quad \forall \vec{\theta} \in \mathbf{X}_h,
\end{aligned}\\
b_{\mathcal{P}}(\phi,\vec{v}_h^n) = 0, \quad \forall \phi \in Q_h.
\end{eqnarray}
\end{subequations}
Thus, we only need to prove the unique solvability of \eqref{eq:CHNS:solvability:P} is equivalent to the unique solvability of \eqref{eq:CHNS:solvability:P_aux}. 
(Necessity) If \eqref{eq:CHNS:solvability:P_aux} has a solution $(y_h^n, \mu_h^n, \vec{v}_h^{n}, p_h^n)$. Define $w_h^n=\mu_h^n-\frac{1}{\abs{\Omega}}(\mu_h^n,1)$, then $(y_h^n, w_h^n, \vec{v}_h^{n}, p_h^n)$ is a solution of \cref{eq:CHNS:solvability:P}. If the solution of \eqref{eq:CHNS:solvability:P_aux} is unique. Assume $(y_h^{n,1}, w_h^{n,1},\vec{v}_h^{n,1},p_h^{n,1})$ and $(y_h^{n,2}, w_h^{n,2},\vec{v}_h^{n,2},p_h^{n,2})$ are two different solutions of \cref{eq:CHNS:solvability:P}, then 
\begin{align*}
\Big(y_h^{n,1},~w_h^{n,1}+\frac{1}{\abs{\Omega}}\big(\Phi_{+}\,\!'(y_h^{n,1}+\bar{c}_0)+\Phi_{-}\,\!'(y_h^{n-1}+\bar{c}_0), 1\big),~\vec{v}_h^{n,1},~p_h^{n,1}\Big) & \quad \text{and}\\ 
\Big(y_h^{n,2},~w_h^{n,2}+\frac{1}{\abs{\Omega}}\big(\Phi_{+}\,\!'(y_h^{n,2}+\bar{c}_0)+\Phi_{-}\,\!'(y_h^{n-1}+\bar{c}_0), 1\big),~\vec{v}_h^{n,2},~p_h^{n,2}\Big) &
\end{align*} 
are two different solutions of \cref{eq:CHNS:solvability:P_aux}. By contradiction argument, we know the solution of \eqref{eq:CHNS:solvability:P} is unique.\\
(Sufficiency) If \eqref{eq:CHNS:solvability:P} has a solution $(y_h^n,w_h^n,\vec{v}_h^n,p_h^n)$. Define $\mu_h^n = w_h^n+\frac{1}{\abs{\Omega}}\big(\Phi_{+}\,\!'(y_h^n+\bar{c}_0)+\Phi_{-}\,\!'(y_h^{n-1}+\bar{c}_0), 1\big)$, then $(y_h^n,\mu_h^n,\vec{v}_h^n,p_h^n)$ is a solution of \eqref{eq:CHNS:solvability:P_aux}. If the solution of \eqref{eq:CHNS:solvability:P} is unique. Assume $(y_h^{n,1}, \mu_h^{n,1}, \vec{v}_h^{n,1}, p_h^{n,1})$ and $(y_h^{n,2}, \mu_h^{n,2}, \vec{v}_h^{n,2}, p_h^{n,2})$ are two different solutions of \cref{eq:CHNS:solvability:P_aux}, then $\big(y_h^{n,1}, \mu_h^{n,1}-\frac{1}{\abs{\Omega}}(\mu_h^{n,1},1),\vec{v}_h^{n,1},p_h^{n,1}\big)$ and $\big(y_h^{n,2}, \mu_h^{n,2}-\frac{1}{\abs{\Omega}}(\mu_h^{n,2},1),\vec{v}_h^{n,1},p_h^{n,1}\big)$ are two different solutions of \cref{eq:CHNS:solvability:P}. This argument is valid, since if $(y_h^n, \mu_h^n, \vec{v}_h^n, p_h^n)$ is a solution of \cref{eq:CHNS:solvability:P_aux}, then $(y_h^n, \mu_h^n+C, \vec{v}_h^n, p_h^n)$ is not a solution of \cref{eq:CHNS:solvability:P_aux}, here $C$ denotes a nonzero constant.
\end{proof}
\begin{theorem}
Based on the auxiliary flow problem \eqref{eq:CHNS:solvability:aux}, the unique solvability of the DG scheme \eqref{eq:CHNS:DGscheme} is equivalent to the unique solvability of the problem \cref{CHNS:solvability:schemeB}.
\end{theorem}
\begin{proof}
By \cref{lem:CHNS:solvability:P}, we only need to prove the unique solvability of \eqref{eq:CHNS:solvability:P} is equivalent to the unique solvability of \cref{CHNS:solvability:schemeB}. From \cref{lem:CHNS:unique_sol_aux_flow}, the auxiliary flow problem \eqref{eq:CHNS:solvability:aux} is always unconditionally unique solvable. Let $(\tilde{\vec{v}}_h^n,\tilde{p}_h^n)$ be the unique solution of \eqref{eq:CHNS:solvability:aux} and we have:\\
(Necessity) If \cref{CHNS:solvability:schemeB} has a solution $(y_h^n,w_h^n,\hat{\vec{v}}_h^n,\hat{p}_h^n)$. Then $(y_h^n,w_h^n,\hat{\vec{v}}_h^n+\tilde{\vec{v}}_h^n,\hat{p}_h^n+\tilde{p}_h^n)$ is a solution of \eqref{eq:CHNS:solvability:P}. 
If the solution of \eqref{CHNS:solvability:schemeB} is unique. Assume $(y_h^{n,1},w_h^{n,1},\vec{v}_h^{n,1},p_h^{n,1})$ and $(y_h^{n,2},w_h^{n,2},\vec{v}_h^{n,2},p_h^{n,2})$ are two different solutions of \eqref{eq:CHNS:solvability:P}. Then $(y_h^{n,1},w_h^{n,1},\vec{v}_h^{n,1}-\tilde{\vec{v}}_h^n,p_h^{n,1}-\tilde{p}_h^n)$ and $(y_h^{n,2},w_h^{n,2},\vec{v}_h^{n,2}-\tilde{\vec{v}}_h^n,p_h^{n,2}-\tilde{p}_h^n)$ are two different solutions of \eqref{CHNS:solvability:schemeB}. By contradiction argument, we know the solution of \eqref{eq:CHNS:solvability:P} is unique.\\
(Sufficiency) If \eqref{eq:CHNS:solvability:P} has a solution $(y_h^n,w_h^n,\vec{v}_h^n,p_h^n)$. Then $(y_h^n,w_h^n,\vec{v}_h^n-\tilde{\vec{v}}_h^n,p_h^n-\tilde{p}_h^n)$ is a solution of \eqref{CHNS:solvability:schemeB}.
If the solution of \eqref{eq:CHNS:solvability:P} is unique. Assume $(y_h^{n,1},w_h^{n,1},\hat{\vec{v}}_h^{n,1},\hat{p}_h^{n,1})$ and $(y_h^{n,2},w_h^{n,2},\hat{\vec{v}}_h^{n,2},\hat{p}_h^{n,2})$ are two different solutions of \eqref{CHNS:solvability:schemeB}. Then $(y_h^{n,1},w_h^{n,1},\hat{\vec{v}}_h^{n,1}+\tilde{\vec{v}}_h^n,\hat{p}_h^{n,1}+\tilde{p}_h^n)$ and $(y_h^{n,2},w_h^{n,2},\hat{\vec{v}}_h^{n,2}+\tilde{\vec{v}}_h^n,\hat{p}_h^{n,2}+\tilde{p}_h^n)$ are two different solutions of \eqref{eq:CHNS:solvability:P}. By contradiction argument, we know the solution of \eqref{CHNS:solvability:schemeB} is unique.
\end{proof}
Now we are in the position to prove \eqref{CHNS:solvability:schemeB} is uniquely solvable. The outline is let us first express $y_h^n$ and $(\vec{v}_h^n,p_h^n)$ in terms of $w_h^n$ by solving \eqref{CHNS:solvability:schemeB:b} and (\ref{CHNS:solvability:schemeB:c}--\ref{CHNS:solvability:schemeB:d}) respectively, then seek the unique existence of solution $w_h^n$.
\begin{lemma}\label{CHNS:solvability:uniq_c}
For each fixed $w_h \in M_h$\,, given $y_h^{n-1} \in M_h$ and $\bar{c}_0 \in S_h$\,, there exists a unique solution $y_h \in M_h$ satisfying
\begin{align}\label{eq:CHNS:solvability:uniq_c}
\big(\Phi_{+}\,\!'(y_h+\bar{c}_0)+\Phi_{-}\,\!'(y_h^{n-1}+\bar{c}_0),\mathring{\varphi}\big) + \kappa a_{\mathcal{D}}(y_h,\mathring{\varphi}) - (w_h,\mathring{\varphi}) = 0, && \forall \mathring{\varphi} \in M_h,
\end{align}
for any mesh size $h$, time step size $\tau$, and parameter $\kappa$.
\end{lemma}
\begin{proof}
We first prove the existence of a solution. For each fixed $w_h \in M_h$\,, define the mapping $\mathcal{F}:~M_h \rightarrow M_h$ by
\begin{multline*}
\big(\mathcal{F}(y_h),\mathring{\varphi}\big) = \big(\Phi_{+}\,\!'(y_h+\bar{c}_0)+\Phi_{-}\,\!'(y_h^{n-1}+\bar{c}_0),\mathring{\varphi}\big)\\ + \kappa a_{\mathcal{D}}(y_h,\mathring{\varphi}) - (w_h,\mathring{\varphi}), \quad
\forall y_h,~\mathring{\varphi} \in M_h.
\end{multline*}
The fact that $\mathcal{F}$ is well defined is guaranteed by the Riesz representation theorem. Taking the Taylor expansion of $\Phi_{+}\,\!'(y_h+\bar{c}_0)$ at $\bar{c}_0$ to first order, there exists $\xi_h$ between $\bar{c}_0$ and $y_h+\bar{c}_0$, such that
\begin{equation*}
\Phi_{+}\,\!'(y_h+\bar{c}_0) = \Phi_{+}\,\!'(\bar{c}_0) + \Phi_{+}\,\!''(\xi_h)y_h.
\end{equation*}
Considering $y_h \in M_h$, and the fact that $\Phi_{+}$\,, the convex part of $\Phi$, satisfies $\Phi_{+}\,\!''\geq0$, we obtain the following inequality
\begin{equation}\label{eq:CHNS:Phi_Taylor2}
\big(\Phi_{+}\,\!'(y_h+\bar{c}_0),y_h\big) = \big(\Phi_{+}\,\!'(\bar{c}_0),y_h\big) + \big(\Phi_{+}\,\!''(\xi_h),y_h^2\big) = \big(\Phi_{+}\,\!''(\xi_h),y_h^2\big) \geq 0.
\end{equation}
We next turn to derive a lower bound of $\big(\mathcal{F}(y_h),y_h\big)$. Applying Cauchy--Schwarz's inequality, Young's inequality, and Poincar\'e's inequality, we have
\begin{align*}
&- \big(\Phi_{-}\,\!'(y_h^{n-1}+\bar{c}_0),y_h\big) + (w_h,y_h)\\
\leq& \norm{\Phi_{-}\,\!'(y_h^{n-1}+\bar{c}_0)}{L^2(\Omega)}\norm{y_h}{L^2(\Omega)} + \norm{w_h}{L^2(\Omega)}\norm{y_h}{L^2(\Omega)}\\
\leq& \frac{C_P^2}{K_\alpha\kappa}\norm{\Phi_{-}\,\!'(y_h^{n-1}+\bar{c}_0)}{L^2(\Omega)}^2 + \frac{K_\alpha\kappa}{4C_P^2}\norm{y_h}{L^2(\Omega)}^2 + \frac{C_P^2}{K_\alpha\kappa}\norm{w_h}{L^2(\Omega)}^2 + \frac{K_\alpha\kappa}{4C_P^2}\norm{y_h}{L^2(\Omega)}^2 \\
\leq& \frac{K_\alpha\kappa}{2}\norm{y_h}{\mathrm{DG}}^2 + \frac{C_P^2}{K_\alpha\kappa}\Big(\norm{\Phi_{-}\,\!'(y_h^{n-1}+\bar{c}_0)}{L^2(\Omega)}^2 + \norm{w_h}{L^2(\Omega)}^2\Big).
\end{align*}
Combining this result with \eqref{eq:CHNS:Phi_Taylor2} and using the coercivity of $a_{\mathcal{D}}$, we 
obtain
\begin{equation*}
\big(\mathcal{F}(y_h),y_h\big) \geq \frac{K_\alpha\kappa}{2} \norm{y_h}{\mathrm{DG}}^2 
- \frac{C_P^2}{K_\alpha\kappa}\Big(\norm{\Phi_{-}\,\!'(y_h^{n-1}+\bar{c}_0)}{L^2(\Omega)}^2 + \norm{w_h}{L^2(\Omega)}^2\Big).
\end{equation*}
Define the sphere $\Xi$ in $M_h$ as follows
\begin{equation*}
\Xi = \Big\{y_h \in M_h:~ \norm{y_h}{\mathrm{DG}}^2 = \frac{2C_P^2}{K_\alpha^2\kappa^2} \Big(\norm{\Phi_{-}\,\!'(y_h^{n-1}+\bar{c}_0)}{L^2(\Omega)}^2 + \norm{w_h}{L^2(\Omega)}^2\Big) \Big\}.
\end{equation*}
We have $\big(\mathcal{F}(y_h),y_h\big)\geq0$ for any $y_h \in \Xi$. By Brouwer's fixed point theorem, there exists a function $y_h \in M_h$ such that $\mathcal{F}(y_h) = 0$. In particular $\big(\mathcal{F}(y_h),\mathring{\varphi}\big) = 0$ for all $\mathring{\varphi} \in M_h$, i.\,e., the function $y_h$ is a solution of \eqref{eq:CHNS:solvability:uniq_c}. 
Next, let us prove the solution of \eqref{eq:CHNS:solvability:uniq_c} is unique. Assume $y_h \in M_h$ and $\tilde{y}_h \in M_h$ are two solutions of \eqref{eq:CHNS:solvability:uniq_c}, then
\begin{align*}
\big(\Phi_{+}\,\!'(y_h+\bar{c}_0)+\Phi_{-}\,\!'(y_h^{n-1}+\bar{c}_0),\,\mathring{\varphi}\big) + \kappa a_{\mathcal{D}}(y_h,\mathring{\varphi}) - (w_h,\mathring{\varphi}) &= 0,\\
\big(\Phi_{+}\,\!'(\tilde{y}_h+\bar{c}_0)+\Phi_{-}\,\!'(y_h^{n-1}+\bar{c}_0),\,\mathring{\varphi}\big) + \kappa a_{\mathcal{D}}(\tilde{y}_h,\mathring{\varphi}) - (w_h,\mathring{\varphi}) &= 0.
\end{align*}
Subtracting above two equations, taking $\mathring{\varphi} = y_h-\tilde{y}_h \in M_h$\,, by the coercivity of $a_{\mathcal{D}}$, we have
\begin{equation*}
K_\alpha\kappa\norm{y_h-\tilde{y}_h}{\mathrm{DG}}^2
\leq \kappa a_{\mathcal{D}}(y_h-\tilde{y}_h,y_h-\tilde{y}_h) \\
= -\big(\Phi_{+}\,\!'(y_h+\bar{c}_0)-\Phi_{+}\,\!'(\tilde{y}_h+\bar{c}_0),y_h-\tilde{y}_h\big).
\end{equation*}
Since $\Phi_{+}$ is convex, then $\Phi_{+}\,\!'$ is non-decreasing, hence we have
\begin{equation*}
\big(\Phi_{+}\,\!'(y_h+\bar{c}_0)-\Phi_{+}\,\!'(\tilde{y}_h+\bar{c}_0)\big) (y_h-\tilde{y}_h) \geq 0.
\end{equation*}
Therefore, we have $\norm{y_h-\tilde{y}_h}{\mathrm{DG}}^2 \leq 0$, which means $\norm{y_h-\tilde{y}_h}{\mathrm{DG}} = 0$. Due to the fact that $\norm{\cdot}{\mathrm{DG}}$ is a norm in $M_h$, we obtain $y_h=\tilde{y}_h$, i.\,e., the solution of \eqref{eq:CHNS:solvability:uniq_c} is unique.
\end{proof}
\begin{lemma}\label{CHNS:solvability:uniq_v_and_p}
For each fixed $w_h \in M_h$\,, given $(y_h^{n-1},\vec{v}_h^{n-1}) \in M_h \times \mathbf{X}_h$ and $\bar{c}_0 \in S_h$\,, there exists a unique solution $(\vec{v}_h, p_h) \in \mathbf{X}_h \times Q_h$ satisfying
\begin{subequations}\label{eq:CHNS:solvability:uniq_v_and_p}
\begin{eqnarray}
\begin{aligned}
\frac{1}{\tau}(\vec{v}_h ,\vec{\theta}) + a_{\mathcal{C}}(\vec{v}_h^{n-1},\vec{v}_h^{n-1},\vec{v}_h,\vec{\theta}) + \mu_\mathrm{s} a_\strain(\vec{v}_h, \vec{\theta}) \hspace*{10em}\\ + b_{\mathcal{P}}(p_h,\vec{\theta}) - b_\mathcal{I}(y_h^{n-1}+\bar{c}_0,w_h,\vec{\theta}) = 0, \quad \forall \vec{\theta} \in \mathbf{X}_h,
\end{aligned}\\
b_{\mathcal{P}}(\phi,\vec{v}_h) = 0, \quad \forall \phi \in Q_h,
\end{eqnarray}
\end{subequations}
for any mesh size $h$, time step size $\tau$, and parameter $\mu_\mathrm{s}$\,.
\end{lemma}
The technique of proving \cref{CHNS:solvability:uniq_v_and_p} is similar to deal with the unique solvability of the auxiliary flow problem \eqref{eq:CHNS:solvability:aux}. With the help of \cref{CHNS:solvability:uniq_c} and \cref{CHNS:solvability:uniq_v_and_p}, we can establish the unconditionally unique solvability of our DG scheme \eqref{eq:CHNS:DGscheme} by invoking the Minty--Browder theorem \cite{ciarlet2013linear}.  
\begin{lemma}
The scheme \eqref{CHNS:solvability:schemeB} is uniquely solvable for any mesh size $h$, time step size $\tau$, parameter $\kappa$, and parameter $\mu_\mathrm{s}$\,.
\end{lemma}
\begin{proof}
For any $w_h \in M_h$, let $y_h$ and $(\vec{v}_h,p_h)$ be the unique solutions of \eqref{CHNS:solvability:schemeB:b} and (\ref{CHNS:solvability:schemeB:c}--\ref{CHNS:solvability:schemeB:d}) which are defined in \cref{CHNS:solvability:uniq_c} and \cref{CHNS:solvability:uniq_v_and_p}, respectively. Define an operator $\mathcal{G}:~M_h\rightarrow M_h'$ (the dual space of $M_h$) as follows
\begin{align*}
\langle\mathcal{G}(w_h),\mathring{\chi}\rangle = (y_h-\hat{y}_h^{n-1},\mathring{\chi}) + \tau a_{\mathcal{D}}(w_h,\mathring{\chi}) + \tau a_{\mathcal{A}}(y_h^{n-1}+\bar{c}_0,\vec{v}_h,\mathring{\chi}), && \forall \mathring{\chi} \in M_h.
\end{align*}
Let us first check the boundedness of $\mathcal{G}$. By triangle inequality, Cauchy--Schwarz's inequality, Poincar\'e's inequality, and the continuity of $a_{\mathcal{D}}$, we have
\begin{align*}
\abs{\langle\mathcal{G}(w_h),\mathring{\chi}\rangle} 
\leq& \norm{y_h}{L^2(\Omega)}\norm{\mathring{\chi}}{L^2(\Omega)} + \norm{\hat{y}_h^{n-1}}{L^2(\Omega)}\norm{\mathring{\chi}}{L^2(\Omega)} \\ &+ \tau \abs{a_{\mathcal{D}}(w_h,\mathring{\chi})} + \tau \abs{a_{\mathcal{A}}(y_h^{n-1}+\bar{c}_0,\vec{v}_h,\mathring{\chi})}\\
\leq& C_P^2\norm{y_h}{\mathrm{DG}}\norm{\mathring{\chi}}{\mathrm{DG}} + C_P\norm{\hat{y}_h^{n-1}}{L^2(\Omega)}\norm{\mathring{\chi}}{\mathrm{DG}}\\ &+ C_\alpha\tau \norm{w_h}{\mathrm{DG}}\norm{\mathring{\chi}}{\mathrm{DG}} + \tau \abs{a_{\mathcal{A}}(y_h^{n-1}+\bar{c}_0,\vec{v}_h,\mathring{\chi})}.
\end{align*}
For the last term in above inequality, considering $y_h^{n-1} \in M_h$, the boundedness of $a_{\mathcal{A}}$ implies
\begin{align}\label{eq:CHNS:solvability_bound_aA}
\begin{split}
&\abs{a_{\mathcal{A}}\big(y_h^{n-1}+\bar{c}_0,\vec{v}_h,\mathring{\chi}\big)} \\
\leq& C_\gamma \Big(\norm{y_h^{n-1}+\bar{c}_0}{\mathrm{DG}} + \abs{\int_{\Omega} (y_h^{n-1}+\bar{c}_0)}\Big)\norm{\vec{v}_h}{\mathrm{DG}}\norm{\mathring{\chi}}{\mathrm{DG}} \\
=& C_\gamma(\norm{y_h^{n-1}+\bar{c}_0}{\mathrm{DG}} + \abs{\Omega}\abs{\bar{c}_0}) \norm{\vec{v}_h}{\mathrm{DG}}\norm{\mathring{\chi}}{\mathrm{DG}},
\end{split}
\end{align}
which means, for any $\mathring{\chi} \in M_h$ with $\norm{\mathring{\chi}}{\mathrm{DG}}=1$, we have
\begin{multline}\label{CHNS:eq:solvability_MB_boundedness1}
\abs{\langle\mathcal{G}(w_h),\mathring{\chi}\rangle} 
\leq C_\alpha\tau \norm{w_h}{\mathrm{DG}} + C_P^2\norm{y_h}{\mathrm{DG}} \\+ C_\gamma\tau (\norm{y_h^{n-1}+\bar{c}_0}{\mathrm{DG}}+\abs{\Omega}\abs{\bar{c}_0})\norm{\vec{v}_h}{\mathrm{DG}} + C_P\norm{\hat{y}_h^{n-1}}{L^2(\Omega)}.
\end{multline}
Our next step is to bound $\norm{y_h}{\mathrm{DG}}$ and $\norm{\vec{v}_h}{\mathrm{DG}}$ by $\norm{w_h}{\mathrm{DG}}$. Since $y_h=y_h(w_h) \in M_h$ is the unique solution of \eqref{eq:CHNS:solvability:uniq_c} which is defined in \cref{CHNS:solvability:uniq_c}, take $\mathring{\varphi}=y_h$ then
\begin{equation}\label{eq:CHNS:solvability:uniq_c_1}
\big(\Phi_{+}\,\!'(y_h+\bar{c}_0)+\Phi_{-}\,\!'(y_h^{n-1}+\bar{c}_0),y_h\big) + \kappa a_{\mathcal{D}}(y_h,y_h) - (w_h,y_h) = 0.
\end{equation}
Recall the nonnegativity of $\big(\Phi_{+}\,\!'(y_h+\bar{c}_0),y_h\big)$ in \eqref{eq:CHNS:Phi_Taylor2}. By the coercivity of $a_{\mathcal{D}}$, Cauchy--Schwarz's inequality and Poincar\'e's inequality, we have
\begin{align*}
K_\alpha\kappa \norm{y_h}{\mathrm{DG}}^2 
\leq& \big(\Phi_{+}\,\!'(y_h+\bar{c}_0),y_h\big) + \kappa a_{\mathcal{D}}(y_h,y_h)\\
=& (w_h,y_h) - (\Phi_{-}\,\!'(y_h^{n-1}+\bar{c}_0),y_h\big)\\
\leq& \norm{w_h}{L^2(\Omega)}\norm{y_h}{L^2(\Omega)} + \norm{\Phi_{-}\,\!'(y_h^{n-1}+\bar{c}_0)}{L^2(\Omega)}\norm{y_h}{L^2(\Omega)}\\
\leq& C_P^2\norm{w_h}{\mathrm{DG}}\norm{y_h}{\mathrm{DG}} + C_P\norm{\Phi_{-}\,\!'(y_h^{n-1}+\bar{c}_0)}{L^2(\Omega)}\norm{y_h}{\mathrm{DG}}.
\end{align*}
Therefore, we obtain the following bound 
\begin{equation}\label{CHNS:eq:solvability_MB_boundedness2}
\norm{y_h}{\mathrm{DG}} \leq \frac{C_P^2}{K_\alpha\kappa}\norm{w_h}{\mathrm{DG}} + \frac{C_P}{K_\alpha\kappa}\norm{\Phi_{-}\,\!'(y_h^{n-1}+\bar{c}_0)}{L^2(\Omega)}.
\end{equation}
Since $(\vec{v}_h,p_h)=\big(\vec{v}_h(w_h),p_h(w_h)\big)$ is the unique solution of \eqref{eq:CHNS:solvability:uniq_v_and_p} which is defined in \cref{CHNS:solvability:uniq_v_and_p}, take $\vec{\theta}=\vec{v}_h$ we have
\[
\frac{1}{\tau}(\vec{v}_h ,\vec{v}_h) + a_{\mathcal{C}}(\vec{v}_h^{n-1},\vec{v}_h^{n-1},\vec{v}_h,\vec{v}_h) + \mu_\mathrm{s} a_\strain(\vec{v}_h, \vec{v}_h)
 - b_\mathcal{I}(y_h^{n-1}+\bar{c}_0,w_h,\vec{v}_h) = 0.
\]
Recall the definition of DG forms $a_{\mathcal{A}}$ and $b_\mathcal{I}$ in \eqref{eq:CHNS:DG_forms}. By the positivity of $a_{\mathcal{C}}$, the coercivity of $a_\strain$, and considering $\vec{v}_h$ is discrete divergence-free, we obtain
\begin{align}\label{eq:CHNS:solvability:a_A}
\begin{split}
&a_{\mathcal{A}}\big(y_h^{n-1}+\bar{c}_0,\vec{v}_h,w_h\big) = b_\mathcal{I}\big(y_h^{n-1}+\bar{c}_0,w_h,\vec{v}_h\big) \\
=& \frac{1}{\tau}(\vec{v}_h,\vec{v}_h) + a_{\mathcal{C}}(\vec{v}_h^{n-1},\vec{v}_h^{n-1},\vec{v}_h,\vec{v}_h) + \mu_\mathrm{s} a_\strain(\vec{v}_h, \vec{v}_h) \\
\geq& K_\strain\mu_\mathrm{s} \norm{\vec{v}_h}{\mathrm{DG}}^2 \geq 0.
\end{split}
\end{align} 
Taking $\mathring{\chi} = w_h$ in \eqref{eq:CHNS:solvability_bound_aA} and combining the result with \eqref{eq:CHNS:solvability:a_A}, we obtain the following bound
\begin{equation}\label{CHNS:eq:solvability_MB_boundedness3}
\norm{\vec{v}_h}{\mathrm{DG}} \leq \frac{C_\gamma}{K_\strain\mu_\mathrm{s}}(\norm{y_h^{n-1}+\bar{c}_0}{\mathrm{DG}} + \abs{\Omega}\abs{\bar{c}_0}) \norm{w_h}{\mathrm{DG}}.
\end{equation}
Substituting \eqref{CHNS:eq:solvability_MB_boundedness2} and \eqref{CHNS:eq:solvability_MB_boundedness3} into \eqref{CHNS:eq:solvability_MB_boundedness1}, we have
\begin{multline*}
\norm{\mathcal{G}(w_h)}{M_h'}
= \sup_{\substack{\forall \mathring{\chi} \in M_h \\ \norm{\mathring{\chi}}{\mathrm{DG}}=1}} \abs{\langle\mathcal{G}(w_h),\mathring{\chi}\rangle}\\
\leq \Big(C_\alpha\tau + \frac{C_P^4}{K_\alpha\kappa} + \frac{C_\gamma^2\tau}{K_\strain \mu_\mathrm{s}} (\norm{y_h^{n-1}+\bar{c}_0}{\mathrm{DG}}+\abs{\Omega}\abs{\bar{c}_0})^2\Big) \norm{w_h}{\mathrm{DG}} \\ + \frac{C_P^3}{K_\alpha\kappa}\norm{\Phi_{-}\,\!'(y_h^{n-1}+\bar{c}_0)}{L^2(\Omega)} + C_P\norm{\hat{y}_h^{n-1}}{L^2(\Omega)}.
\end{multline*}
Due to the fact that $y_h^{n-1}, \hat{y}_h^{n-1} \in M_h$ and $\bar{c}_0 \in S_h$ are given quantities, the above inequality shows that the operator $\mathcal{G}$ maps bounded sets in $M_h$ to bounded sets in $M_h'$, i.\,e., we have proved boundedness of the operator.
Second, we show the coercivity of $\mathcal{G}$. By Cauchy--Schwarz's inequality and Poincar\'e's inequality, we have
\begin{align}\label{eq:CHNS:solvability:boundness1}
(\hat{y}_h^{n-1},w_h) 
\leq \norm{\hat{y}_h^{n-1}}{L^2(\Omega)}\norm{w_h}{L^2(\Omega)} 
\leq C_p\norm{\hat{y}_h^{n-1}}{L^2(\Omega)}\norm{w_h}{\mathrm{DG}}.
\end{align}
We again use \eqref{eq:CHNS:solvability:uniq_c_1} and \eqref{eq:CHNS:Phi_Taylor2}. By the coercivity of $a_{\mathcal{D}}$, Cauchy--Schwarz's inequality, Young's inequality, and Poincar\'e's inequality, we have
\begin{align}\label{eq:CHNS:solvability:boundness2}
\begin{split}
-(w_h,y_h) 
\leq& - \big(\Phi_{-}\,\!'(y_h^{n-1}+\bar{c}_0),y_h\big) - \kappa a_{\mathcal{D}}(y_h,y_h)\\
\leq& \norm{\Phi_{-}\,\!'(y_h^{n-1}+\bar{c}_0)}{L^2(\Omega)}\norm{y_h}{L^2(\Omega)} - K_\alpha\kappa \norm{y_h}{\mathrm{DG}}^2\\
\leq& \frac{C_P^2}{4K_\alpha \kappa}\norm{\Phi_{-}\,\!'(y_h^{n-1}+\bar{c}_0)}{L^2(\Omega)}^2 + \frac{K_\alpha \kappa}{C_P^2}\norm{y_h}{L^2(\Omega)}^2 - K_\alpha\kappa \norm{y_h}{\mathrm{DG}}^2\\
\leq& \frac{C_P^2}{4K_\alpha \kappa}\norm{\Phi_{-}\,\!'(y_h^{n-1}+\bar{c}_0)}{L^2(\Omega)}^2.
\end{split}
\end{align}
Using the definition of $\mathcal{G}$, the coercivity of $a_{\mathcal{D}}$, the bounds \eqref{eq:CHNS:solvability:a_A}, \eqref{eq:CHNS:solvability:boundness1}, and \eqref{eq:CHNS:solvability:boundness2}, we obtain
\begin{align*}
\langle\mathcal{G}(w_h),w_h\rangle =& (y_h-\hat{y}_h^{n-1},w_h) + \tau a_{\mathcal{D}}(w_h,w_h) + \tau a_{\mathcal{A}}(y_h^{n-1}+\bar{c}_0,\vec{v}_h,w_h)\\
\geq& K_\alpha\tau \norm{w_h}{\mathrm{DG}}^2 -C_p\norm{\hat{y}_h^{n-1}}{L^2(\Omega)}\norm{w_h}{\mathrm{DG}} -\frac{C_P^2}{4K_\alpha\kappa}\norm{\Phi_{-}\,\!'(y_h^{n-1}+\bar{c}_0)}{L^2(\Omega)}^2. 
\end{align*}
Since $y_h^{n-1}, \hat{y}_h^{n-1} \in M_h$ and $\bar{c}_0 \in S_h$ are given quantities, it is obvious that
\begin{equation*}
\lim_{\norm{w_h}{\mathrm{DG}}\rightarrow +\infty} \frac{\langle\mathcal{G}(w_h),w_h\rangle}{\norm{w_h}{\mathrm{DG}}}= +\infty.
\end{equation*}
Therefore we proved the coercivity of $\mathcal{G}$.
Third, let us check the monotonicity of $\mathcal{G}$. For any $w_h$ and $s_h$ in $M_h$, we have
\begin{align}\label{eq:CHNS:solvability:monotonicity1}
\begin{split}
&\langle\mathcal{G}(w_h)-\mathcal{G}(s_h),w_h-s_h\rangle\\
=& \langle\mathcal{G}(w_h),w_h\rangle - \langle\mathcal{G}(w_h),s_h\rangle - \langle\mathcal{G}(s_h),w_h\rangle + \langle\mathcal{G}(s_h),s_h\rangle \\
=& (y_h(w_h)-y_h(s_h),w_h-s_h) + \tau a_{\mathcal{D}}(w_h-s_h,w_h-s_h)\\ &+ \tau a_{\mathcal{A}}(y_h^{n-1}+\bar{c}_0,\vec{v}_h(w_h)-\vec{v}_h(s_h),w_h-s_h).
\end{split}
\end{align}
Due to the coercivity of $a_{\mathcal{D}}$, the second term above is always nonnegative, which means we only need to check the sign of the first and the third terms. From \cref{CHNS:solvability:uniq_c}, for any $\mathring{\varphi} \in M_h$, we obtain
\begin{align*}
(w_h,\mathring{\varphi}) &= \big(\Phi_{+}\,\!'(y_h(w_h)+\bar{c}_0)+\Phi_{-}\,\!'(y_h^{n-1}+\bar{c}_0),\mathring{\varphi}\big) + \kappa a_{\mathcal{D}}(y_h(w_h),\mathring{\varphi}),\\
\,\,(s_h,\mathring{\varphi}) &= \big(\Phi_{+}\,\!'(y_h(s_h)+\bar{c}_0)+\Phi_{-}\,\!'(y_h^{n-1}+\bar{c}_0),\mathring{\varphi}\big) + \kappa a_{\mathcal{D}}(y_h(s_h),\mathring{\varphi}).
\end{align*}
Subtracting the two equations above, for any $\mathring{\varphi} \in M_h$, we have
\begin{multline*}
(w_h-s_h,\mathring{\varphi}) = \big(\Phi_{+}\,\!'(y_h(w_h)+\bar{c}_0)-\Phi_{+}\,\!'(y_h(s_h)+\bar{c}_0),\mathring{\varphi}\big)\\ + \kappa a_{\mathcal{D}}(y_h(w_h)-y_h(s_h),\mathring{\varphi}).
\end{multline*}
By \cref{CHNS:solvability:uniq_c}, we know that $y_h(w_h)$ and $y_h(s_h)$ belong to $M_h$. We may then choose $\mathring{\varphi} = y_h(w_h)-y_h(s_h) \in M_h$ in the equation above. Using the fact that $\Phi_{+}\,\!'$ is non-decreasing and the coercivity of $a_{\mathcal{D}}$, we obtain
\begin{align}\label{eq:CHNS:solvability:monotonicity2}
\begin{split}
&\big(y_h(w_h)-y_h(s_h),w_h-s_h\big)\\
=& \big(\Phi_{+}\,\!'(y_h(w_h)+\bar{c}_0)-\Phi_{+}\,\!'(y_h(s_h)+\bar{c}_0),y_h(w_h)-y_h(s_h)\big) \\ &+ \kappa a_{\mathcal{D}}(y_h(w_h)-y_h(s_h),y_h(w_h)-y_h(s_h))\\ 
\geq& K_\alpha\kappa \norm{y_h(w_h)-y_h(s_h)}{\mathrm{DG}}^2 \geq 0.
\end{split}
\end{align}
From \cref{CHNS:solvability:uniq_v_and_p}, for any $\vec{\theta} \in \mathbf{X}_h$, we obtain
\begin{align*}
b_\mathcal{I}(y_h^{n-1}+\bar{c}_0,w_h,\vec{\theta}) =& \frac{1}{\tau}\big(\vec{v}_h(w_h),\vec{\theta}\big) + a_{\mathcal{C}}\big(\vec{v}_h^{n-1},\vec{v}_h^{n-1},\vec{v}_h(w_h),\vec{\theta}\big) \\ &+ \mu_\mathrm{s} a_\strain\big(\vec{v}_h(w_h),\vec{\theta}\big) + b_{\mathcal{P}}\big(p_h(w_h),\vec{\theta}\big),\\
b_\mathcal{I}(y_h^{n-1}+\bar{c}_0,s_h,\vec{\theta}) =& \frac{1}{\tau}\big(\vec{v}_h(s_h),\vec{\theta}\big) + a_{\mathcal{C}}\big(\vec{v}_h^{n-1},\vec{v}_h^{n-1},\vec{v}_h(s_h),\vec{\theta}\big) \\ &+ \mu_\mathrm{s} a_\strain\big(\vec{v}_h(s_h),\vec{\theta}\big) + b_{\mathcal{P}}\big(p_h(s_h),\vec{\theta}\big).
\end{align*}
Subtracting the two equations above, for any $\vec{\theta} \in \mathbf{X}_h$, we have
\begin{multline*}
b_\mathcal{I}(y_h^{n-1}+\bar{c}_0,w_h-s_h,\vec{\theta}) 
= \frac{1}{\tau}\big(\vec{v}_h(w_h)-\vec{v}_h(s_h),\vec{\theta}\big) \\
+ a_{\mathcal{C}}\big(\vec{v}_h^{n-1},\vec{v}_h^{n-1},\vec{v}_h(w_h)-\vec{v}_h(s_h),\vec{\theta}\big)\\
+\mu_\mathrm{s} a_\strain\big(\vec{v}_h(w_h)-\vec{v}_h(s_h),\vec{\theta}\big) + b_{\mathcal{P}}\big(p_h(w_h)-p_h(s_h),\vec{\theta}\big).
\end{multline*}
We may then choose $\vec{\theta}=\vec{v}_h(w_h)-\vec{v}_h(s_h) \in \mathbf{X}_h$ in the equation above. Using the positivity of $a_{\mathcal{C}}$, the coercivity of $a_\strain$, considering $\vec{v}_h(w_h)$ and $\vec{v}_h(s_h)$ are discretely divergence-free, and by \cref{rem:CHNS:relation_aA_bI}, we obtain 
\begin{align}\label{eq:CHNS:solvability:monotonicity3}
\begin{split}
& a_{\mathcal{A}}\big(y_h^{n-1}+\bar{c}_0,\vec{v}_h(w_h)-\vec{v}_h(s_h),w_h-s_h\big) \\
=\,& b_\mathcal{I}\big(y_h^{n-1}+\bar{c}_0,w_h-s_h,\vec{v}_h(w_h)-\vec{v}_h(s_h)\big) \\
=\,& \frac{1}{\tau}\big(\vec{v}_h(w_h)-\vec{v}_h(s_h),\vec{v}_h(w_h)-\vec{v}_h(s_h)\big)\\ 
& + a_{\mathcal{C}}\big(\vec{v}_h^{n-1},\vec{v}_h^{n-1},\vec{v}_h(w_h)-\vec{v}_h(s_h),\vec{v}_h(w_h)-\vec{v}_h(s_h)\big) \\
& + \mu_\mathrm{s} a_\strain\big(\vec{v}_h(w_h)-\vec{v}_h(s_h),\vec{v}_h(w_h)-\vec{v}_h(s_h)\big)\\
& + b_{\mathcal{P}}\big(p_h(w_h)-p_h(s_h),\vec{v}_h(w_h)-\vec{v}_h(s_h)\big) \\
\geq\,& K_\strain\mu_\mathrm{s} \norm{\vec{v}_h(w_h)-\vec{v}_h(s_h)}{\mathrm{DG}}^2 \geq 0.
\end{split}
\end{align}
Substituting \eqref{eq:CHNS:solvability:monotonicity2} and \eqref{eq:CHNS:solvability:monotonicity3} into \eqref{eq:CHNS:solvability:monotonicity1}, considering $\norm{\cdot}{\mathrm{DG}}$ is a norm in $M_h$, the following inequality is strict whenever $w_h \neq s_h$, i.\,e.,
\begin{equation*}
\langle\mathcal{G}(w_h)-\mathcal{G}(s_h),w_h-s_h\rangle 
\geq K_\alpha\tau\norm{w_h-s_h}{\mathrm{DG}}^2
\geq 0.
\end{equation*}
Thus we establish the strict monotonicity of $\mathcal{G}$.
Finally, let us show the continuity of $\mathcal{G}$. For any $\mathring{\chi} \in M_h$ with $\norm{\mathring{\chi}}{\mathrm{DG}}=1$, by triangle inequality, Cauchy--Schwarz's inequality, the continuity of $a_{\mathcal{D}}$, the boundedness of $a_{\mathcal{A}}$, and Poincar\'e's inequality, we have
\begin{align}\label{eq:CHNS:solvability:continuity1}
\begin{split}
&\abs{\langle\mathcal{G}(w_h)-\mathcal{G}(s_h),\mathring{\chi}\rangle}\\
\leq& \abs{(y_h(w_h)-y_h(s_h),\mathring{\chi})} + \tau \abs{a_{\mathcal{D}}(w_h-s_h,\mathring{\chi})} \\ &+ \tau \abs{a_{\mathcal{A}}(y_h^{n-1}+\bar{c}_0,\vec{v}_h(w_h)-\vec{v}_h(s_h),\mathring{\chi})}\\
\leq& \norm{y_h(w_h)-y_h(s_h)}{L^2(\Omega)}\norm{\mathring{\chi}}{L^2(\Omega)} + C_\alpha\tau \norm{w_h-s_h}{\mathrm{DG}}\norm{\mathring{\chi}}{\mathrm{DG}} \\ &+ C_\gamma\tau(\norm{y_h^{n-1}+\bar{c}_0}{\mathrm{DG}} + \abs{\Omega}\abs{\bar{c}_0}) \norm{\vec{v}_h(w_h)-\vec{v}_h(s_h)}{\mathrm{DG}}\norm{\mathring{\chi}}{\mathrm{DG}}\\
\leq& C_\alpha\tau \norm{w_h-s_h}{\mathrm{DG}} + C_P^2\norm{y_h(w_h)-y_h(s_h)}{\mathrm{DG}} \\ &+ C_\gamma\tau (\norm{y_h^{n-1}+\bar{c}_0}{\mathrm{DG}} + \abs{\Omega}\abs{\bar{c}_0})\norm{\vec{v}_h(w_h)-\vec{v}_h(s_h)}{\mathrm{DG}}.
\end{split}
\end{align}
We now estimate the second term above. By \eqref{eq:CHNS:solvability:monotonicity2}, Cauchy--Schwarz's inequality, and Poincar\'e's inequality, we obtain
\begin{align*}
K_\alpha\kappa\norm{y_h(w_h)-y_h(s_h)}{\mathrm{DG}}^2 \leq& \big(y_h(w_h)-y_h(s_h),w_h-s_h\big)\\
\leq& \norm{y_h(w_h)-y_h(s_h)}{L^2(\Omega)}\norm{w_h-s_h}{L^2(\Omega)}\\
\leq& C_P^2\norm{y_h(w_h)-y_h(s_h)}{\mathrm{DG}}\norm{w_h-s_h}{\mathrm{DG}},
\end{align*}
which implies the following bound
\begin{equation}\label{eq:CHNS:solvability:continuity2}
\norm{y_h(w_h)-y_h(s_h)}{\mathrm{DG}} \leq \frac{C_P^2}{K_\alpha\kappa}\norm{w_h-s_h}{\mathrm{DG}}.
\end{equation}
Similarly, by \eqref{eq:CHNS:solvability:monotonicity3} and the boundedness of $a_{\mathcal{A}}$, we have
\begin{multline*}
K_\strain\mu_\mathrm{s}\norm{\vec{v}_h(w_h)-\vec{v}_h(s_h)}{\mathrm{DG}}^2 
\leq a_{\mathcal{A}}\big(y_h^{n-1}+\bar{c}_0,\vec{v}_h(w_h)-\vec{v}_h(s_h),w_h-s_h\big)\\
\leq C_\gamma \big(\norm{y_h^{n-1}+\bar{c}_0}{\mathrm{DG}}+\abs{\Omega}\abs{\bar{c}_0}\big) \norm{\vec{v}_h(w_h)-\vec{v}_h(s_h)}{\mathrm{DG}}\norm{w_h-s_h}{\mathrm{DG}},
\end{multline*}
which implies the following bound
\begin{equation}\label{eq:CHNS:solvability:continuity3}
\norm{\vec{v}_h(w_h)-\vec{v}_h(s_h)}{\mathrm{DG}} \leq \frac{C_\gamma}{K_\strain\mu_\mathrm{s}} (\norm{y_h^{n-1}+\bar{c}_0}{\mathrm{DG}}+\abs{\Omega}\abs{\bar{c}_0}) \norm{w_h-s_h}{\mathrm{DG}}.
\end{equation}
Combining \eqref{eq:CHNS:solvability:continuity2}, \eqref{eq:CHNS:solvability:continuity3}, and \eqref{eq:CHNS:solvability:continuity1}, we obtain 
\begin{multline*}
\norm{\mathcal{G}(w_h)-\mathcal{G}(s_h)}{M_h'} 
= \sup_{\substack{\forall \mathring{\chi} \in M_h \\ \norm{\mathring{\chi}}{\mathrm{DG}}=1}} \abs{\langle\mathcal{G}(w_h)-\mathcal{G}(s_h),\mathring{\chi}\rangle}\\
\leq \Big(C_\alpha\tau + \frac{C_P^4}{K_\alpha\kappa} + \frac{C_\gamma^2 \tau}{K_\strain \mu_\mathrm{s}}(\norm{y_h^{n-1}+\bar{c}_0}{\mathrm{DG}}+\abs{\Omega}\abs{\bar{c}_0})^2\Big)\norm{w_h-s_h}{\mathrm{DG}},
\end{multline*}
which means $\norm{\mathcal{G}(w_h)-\mathcal{G}(s_h)}{M_h'}$ tends to zero whenever $\norm{w_h-s_h}{\mathrm{DG}}$ tends to zero, i.\,e., we proved the continuity of the operator $\mathcal{G}$.
All conditions of the Minty--Browder theorem are satisfied. We conclude that there exists a unique solution $w_h^n$ such that $\langle\mathcal{G}(w_h^n),\mathring{\chi}\rangle = 0$ for all $\mathring{\chi} \in M_h$. Recall \cref{CHNS:solvability:uniq_c} and \cref{CHNS:solvability:uniq_v_and_p}, this implies that $\big(y_h(w_h^n),w_h^n,\vec{v}_h(w_h^n),p_h(w_h^n)\big)$ is the unique solution of scheme \eqref{CHNS:solvability:schemeB}.
\end{proof}
Therefore we have proved the following result.
\begin{theorem}\label{thm:CHNS:unique_solvability}
The DG scheme \eqref{eq:CHNS:DGscheme} is uniquely solvable for any mesh size $h$, time step size $\tau$, parameter $\kappa$, and parameter $\mu_\mathrm{s}$\,.
\end{theorem}
\begin{remark}
It is easy to check that \cref{thm:CHNS:unique_solvability} is valid for non-symmetric version of the discontinuous Galerkin formulation as well.
\end{remark}

\subsection{Stability analysis}\label{sec:CHNS:DG_stability}
In this section, we show the discrete solution of \eqref{eq:CHNS:DGscheme} satisfies the energy dissipation property and we derive stability bounds valid for any chemical energy density $\Phi$. Analoguously to  the energy \eqref{eq:CHNS:energy} at the  continuous level, we define the discrete energy:
\begin{equation}\label{eq:CHNS:discrete_energy}
F_h(c_h,\vec{v}_h) = \frac{1}{2}(\vec{v}_h,\vec{v}_h) + \big(\Phi(c_h),1\big) + \frac{\kappa}{2}a_{\mathcal{D}}(c_h,c_h).
\end{equation}
The next statement, the discrete energy dissipation law, stems directly from the positivity in \cref{lem:CHNS:positivity_convection} and the convex-concave splitting.
\begin{theorem}\label{thm:CHNS:dis_energy_dissipation}
Let $(c_h^n, \mu_h^n, \vec{v}_h^{n}, p_h^n) \in S_h \times S_h \times \mathbf{X}_h \times Q_h$ be the unique solution of the DG scheme \eqref{eq:CHNS:DGscheme}. Then for any mesh size $h$, time step size $\tau$, parameter $\kappa$\,, and parameter $\mu_\mathrm{s}$\,, the discrete energy \eqref{eq:CHNS:discrete_energy} is non-increasing in time.
\[
\forall 1\leq n \leq N,\quad  F_h(c_h^n, \vec{v}_h^n) \leq F_h(c_h^{n-1}, \vec{v}_h^{n-1}).
\]
\end{theorem}
\begin{proof}
Take $\chi=\mu_h^n$ in \eqref{eq:CHNS:DGscheme:a}, $\varphi=\delta_\tau c_h^n$ in \eqref{eq:CHNS:DGscheme:b}, $\vec{\theta}=\vec{v}_h^n$ in \eqref{eq:CHNS:DGscheme:c}, and $\phi=-p_h^n$ in \eqref{eq:CHNS:DGscheme:d}:
\begin{eqnarray*}
(\delta_\tau c_h^n,\mu_h^n) + a_{\mathcal{D}}(\mu_h^n,\mu_h^n) + a_{\mathcal{A}}(c_h^{n-1},\vec{v}_h^n,\mu_h^n) = 0, \\
(\Phi_{+}\,\!'(c_h^n)+\Phi_{-}\,\!'(c_h^{n-1}),\,\delta_\tau c_h^n) + \kappa a_{\mathcal{D}}(c_h^n,\delta_\tau c_h^n) - (\mu_h^n,\delta_\tau c_h^n) = 0, \\
(\delta_\tau \vec{v}_h^n ,\vec{v}_h^n) + a_{\mathcal{C}}(\vec{v}_h^{n-1},\vec{v}_h^{n-1},\vec{v}_h^n,\vec{v}_h^n) + \mu_\mathrm{s} a_\strain(\vec{v}_h^n, \vec{v}_h^n) \hspace*{5.15em}\\ + b_{\mathcal{P}}(p_h^n,\vec{v}_h^n) - b_\mathcal{I}(c_h^{n-1},\mu_h^n,\vec{v}_h^n) = 0, \\
-b_{\mathcal{P}}(p_h^n,\vec{v}_h^n) = 0.
\end{eqnarray*}
Adding the equations above, considering \cref{rem:CHNS:relation_aA_bI}, and benefitting from the positivity of $a_\mathcal{C}$, the coercivity of $a_\mathcal{D}$ and $a_\strain$, we have
\begin{align}\label{eq:CHNS:Dis_energy_1}
\begin{split}
&(\delta_\tau \vec{v}_h^n ,\vec{v}_h^n) + \big(\Phi_{+}\,\!'(c_h^n)+\Phi_{-}\,\!'(c_h^{n-1}),\,\delta_\tau c_h^n\big) + \kappa a_{\mathcal{D}}(c_h^n,\delta_\tau c_h^n) \\
=& -a_{\mathcal{D}}(\mu_h^n,\mu_h^n) - a_{\mathcal{C}}(\vec{v}_h^{n-1},\vec{v}_h^{n-1},\vec{v}_h^n,\vec{v}_h^n) - \mu_\mathrm{s} a_\strain(\vec{v}_h^n, \vec{v}_h^n)\\ &- a_{\mathcal{A}}(c_h^{n-1},\vec{v}_h^n,\mu_h^n) + b_\mathcal{I}(c_h^{n-1},\mu_h^n,\vec{v}_h^n) \\
\leq& -a_{\mathcal{D}}(\mu_h^n,\mu_h^n) - \mu_\mathrm{s} a_\strain(\vec{v}_h^n, \vec{v}_h^n)
\leq -K_\alpha\norm{\mu_h^n}{\mathrm{DG}}^2 - K_\strain\mu_\mathrm{s}\norm{\vec{v}_h^n}{\mathrm{DG}}^2 
\leq 0.
\end{split}
\end{align}
For the term $(\Phi_{+}\,\!'(c_h^n)+\Phi_{-}\,\!'(c_h^{n-1}),\,\delta_\tau c_h^n)$\,, we utilize Taylor expansions up to the second order. There exist $\xi_h$ and $\eta_h$ between $c_h^{n-1}$ and $c_h^n$ such that
\begin{align*}
\Phi_+\,\!'(c_h^n)(c_h^n-c_h^{n-1}) &= \Phi_+(c_h^n) - \Phi_+(c_h^{n-1}) + \frac{1}{2}\Phi_{+}\,\!''(\xi_h)(c_h^{n-1}-c_h^n)^2,\\
\Phi_{-}\,\!'(c_h^{n-1})(c_h^n-c_h^{n-1}) &= \Phi_-(c_h^n) - \Phi_-(c_h^{n-1}) - \frac{1}{2}\Phi_{-}\,\!''(\eta_h)(c_h^n-c_h^{n-1})^2.
\end{align*}
Adding above two equations and using the fact that $\Phi_+$ is convex and $\Phi_-$ is concave, we have
\begin{align}\label{eq:CHNS:Dis_energy_2}
\begin{split}
&\big(\Phi_{+}\,\!'(c_h^n)+\Phi_{-}\,\!'(c_h^{n-1}),\,\delta_\tau c_h^n\big) \\
=& \big(\delta_\tau\Phi(c_h^n),\,1\big) + \frac{1}{2\tau}\big(\Phi_{+}\,\!''(\xi_h), (c_h^{n-1}-c_h^n)^2\big) - \frac{1}{2\tau}\big(\Phi_{-}\,\!''(\eta_h), (c_h^n-c_h^{n-1})^2\big)\\
\geq& \big(\delta_\tau\Phi(c_h^n),\,1\big).
\end{split}
\end{align}
For the terms $(\delta_\tau \vec{v}_h^n ,\vec{v}_h^n)$ and $\kappa a_{\mathcal{D}}(c_h^n,\delta_\tau c_h^n)$, since the inner product and $a_{\mathcal{D}}$ are both symmetric bilinear forms, we immediately have
\begin{align}
(\delta_\tau \vec{v}_h^n ,\vec{v}_h^n) 
\geq& \frac{1}{2\tau}(\vec{v}_h^n,\vec{v}_h^n) - \frac{1}{2\tau}(\vec{v}_h^{n-1},\vec{v}_h^{n-1}), \label{eq:CHNS:Dis_energy_3}\\
a_{\mathcal{D}}(c_h^n,\delta_\tau c_h^n) 
\geq& \frac{1}{2\tau}a_{\mathcal{D}}(c_h^n,c_h^n) - \frac{1}{2\tau}a_{\mathcal{D}}(c_h^{n-1},c_h^{n-1}). \label{eq:CHNS:Dis_energy_4}
\end{align}
Combine \eqref{eq:CHNS:Dis_energy_1} -- \eqref{eq:CHNS:Dis_energy_4} together and recall the definition of the discrete energy \eqref{eq:CHNS:discrete_energy}, then
\begin{align*}
0 \geq& -K_\alpha\norm{\mu_h^n}{\mathrm{DG}}^2 - K_\strain\mu_\mathrm{s}\norm{\vec{v}_h^n}{\mathrm{DG}}^2\\
\geq& \frac{1}{2\tau}(\vec{v}_h^n,\vec{v}_h^n) - \frac{1}{2\tau}(\vec{v}_h^{n-1},\vec{v}_h^{n-1}) + \frac{1}{\tau}\big(\Phi(c_h^n),\,1\big) \\ &- \frac{1}{\tau}\big(\Phi(c_h^{n-1}),\,1\big) + \frac{\kappa}{2\tau}a_{\mathcal{D}}(c_h^n,c_h^n) - \frac{\kappa}{2\tau}a_{\mathcal{D}}(c_h^{n-1},c_h^{n-1})\\
=& \frac{1}{\tau}F_h(c_h^n,\vec{v}_h^n) - \frac{1}{\tau}F_h(c_h^{n-1},\vec{v}_h^{n-1}),
\end{align*}
which means the discrete energy $F_h(c_h,\vec{v}_h)$ is non-increasing in time.
\end{proof}
Throughout the paper, the constant $C$ denotes a generic constant that takes different values at different places and that is independent of
$h$ and $\tau$. 
It is reasonable to assume the initial energy $F_h(c_h^0, \vec{v}_h^0)$ is finite. The following a priori bounds for the order parameter, chemical potential and velocity are a direct result of the discrete energy dissipation law (\cref{thm:CHNS:dis_energy_dissipation}). 
\begin{theorem}\label{thm:CHNS:stability_bound1}
Let $(c_h^n, \mu_h^n, \vec{v}_h^{n}, p_h^n) \in S_h \times S_h \times \mathbf{X}_h \times Q_h$ be the unique solution of the DG scheme \eqref{eq:CHNS:DGscheme}. Then for any mesh size $h$, time step size $\tau$, parameter $\kappa$, and parameter $\mu_\mathrm{s}$, and for any $1 \leq \ell \leq N$ we have
\begin{multline}\label{eq:CHNS:Stability}
\frac{1}{2}\norm{\vec{v}_h^\ell}{L^2(\Omega)}^2 + \big(\Phi(c_h^\ell),1\big) + \frac{K_\alpha\kappa}{2}\norm{c_h^\ell}{\mathrm{DG}}^2 \\ + \tau K_\alpha\sum_{n=1}^\ell \norm{\mu_h^n}{\mathrm{DG}}^2 + \tau K_\strain\mu_\mathrm{s}\sum_{n=1}^\ell \norm{\vec{v}_h^n}{\mathrm{DG}}^2 \leq F_h(c_h^0, \vec{v}_h^0).
\end{multline}
In addition, if the chemical energy density $\Phi$ is bounded from below by a constant (not necessarily positive), as it is the case for
the Ginzburg--Landau double well potential or the logarithmic potential in \eqref{eq:CHNS:chemical_energy_density}, then there is a positive constant $C$
independent of $h$ and $\tau$ such that
\begin{subequations}\label{eq:CHNS:stability_result}
\begin{align}
\max_{1\leq n\leq \ell}\norm{c_h^n}{\mathrm{DG}}^2 + \max_{1\leq n\leq \ell}\norm{\vec{v}_h^n}{L^2(\Omega)}^2 &\leq C, \label{eq:CHNS:stability_result1}\\
\tau\sum_{n=1}^\ell \norm{\mu_h^n}{\mathrm{DG}}^2 + \tau\sum_{n=1}^\ell \norm{\vec{v}_h^n}{\mathrm{DG}}^2 &\leq C. \label{eq:CHNS:stability_result2}
\end{align}
\end{subequations}
\end{theorem}
\begin{proof}
From the proof of \cref{thm:CHNS:dis_energy_dissipation}, we know that
\begin{equation*}
\tau K_\alpha\norm{\mu_h^n}{\mathrm{DG}}^2 + \tau K_\strain\mu_\mathrm{s}\norm{\vec{v}_h^n}{\mathrm{DG}}^2 \leq F_h(c_h^{n-1},\vec{v}_h^{n-1}) - F_h(c_h^n,\vec{v}_h^n).
\end{equation*}
For any $1 \leq \ell \leq N$\,, take the summation of $n$ from $1$ to $\ell$, then
\begin{equation*}
\tau K_\alpha\sum_{n=1}^\ell \norm{\mu_h^n}{\mathrm{DG}}^2 + \tau K_\strain\mu_\mathrm{s} \sum_{n=1}^\ell \norm{\vec{v}_h^n}{\mathrm{DG}}^2 \leq F_h(c_h^0, \vec{v}_h^0) - F_h(c_h^\ell, \vec{v}_h^\ell).
\end{equation*}
Finally \eqref{eq:CHNS:Stability} is obtained by moving $F_h(c_h^\ell, \vec{v}_h^\ell)$ to the left-hand side and using the coercivity of $a_{\mathcal{D}}$. In case $\Phi$ is bounded from below by a constant, since the parameters $\kappa$, $\mu_\mathrm{s}$ and constants $K_\alpha$, $K_\strain$ are all positive, it is straightforward to show \eqref{eq:CHNS:stability_result} holds.
\end{proof}

\subsection{Error analysis}\label{sec:CHNS:error_analysis}
In this section, we derive an optimal error estimate for the fully discrete scheme \eqref{eq:CHNS:DGscheme} in terms of time and space discretization parameters. 
We show that the method \eqref{eq:CHNS:DGscheme} converges for any general chemical energy density that satisfies Lipschitz continuity constraints on the first order derivative of the convex and concave decomposition. More precisely, the Lipschitz condition is as follows:
\begin{assumption}\label{as:CHNS:assumptionA}
There is a constant $C_\mathrm{lip}>0$ independent of mesh size $h$ and time step size $\tau$ such that for all $n\geq 0$
\begin{align}\label{eq:CHNS:error_assumption_on_Phi}
\begin{split}
\norm{\Phi_+\,\!'(c_h^n)-\Phi_+\,\!'(c^n)}{\mathrm{DG}} &\leq C_\mathrm{lip}\norm{c_h^n-c^n}{\mathrm{DG}},\\
\norm{\Phi_-\,\!'(c_h^n)-\Phi_-\,\!'(c^n)}{\mathrm{DG}} &\leq C_\mathrm{lip}\norm{c_h^n-c^n}{\mathrm{DG}}.
\end{split}
\end{align}
In addition, we also assume that both $\Phi_-\,\!''$ and $\Phi_-\,\!'''$ are bounded.
\end{assumption} 
\begin{remark}
\cref{as:CHNS:assumptionA} with respect to $\Phi_-$ is automatically satisfied if $\Phi$ is the  Ginzburg--Landau potential \eqref{eq:CHNS:chemical_energy_density:GL} or the logarithmic potential \eqref{eq:CHNS:chemical_energy_density:Log}. Checking that the assumption \cref{eq:CHNS:error_assumption_on_Phi} holds for $\Phi_+$ is a complicated task in general. However, it is   common practice to truncate the potential and to define an extension such that \cref{as:CHNS:assumptionA} is  satisfied for any general potential function. 
\end{remark}
We also assume that the weak solutions are regular enough. More precisely, we have 
\begin{subequations}\label{eq:CHNS:error_regularities}
\begin{align}
c, \partial_t c &\in L^\infty(0,\, T;\, H^{q+1}(\Omega)), \hspace*{0.45cm}
\partial_{tt} c \in L^2(0,\, T;\, L^2(\Omega)),\\
\mu &\in L^{\infty}\big(0,\, T;\, H^{q+1}(\Omega)\big)\cap L^{\infty}(0,T;W^{1,4}(\Omega)),\\
\vec{v} &\in L^{\infty}(0,\, T;\, H^{q+1}(\Omega)^d),\quad
\partial_t \vec{v} \in L^2(0,T;\, H^q(\Omega)^d),\\
\partial_{tt} \vec{v} &\in L^2(0,\, T;\, L^2(\Omega)^d), \hspace*{1.27cm}
p \in  L^\infty(0,\, T;\, H^{q}(\Omega)).
\end{align}
\end{subequations}
For simplicity, we denote by $c^n,\mu^n,\vec{v}^n$, and $p^n$ the functions $c,\mu,\vec{v}$, and $p$ evaluated at $t^n$\,. With regularities \eqref{eq:CHNS:error_regularities}, it is straightforward to check that, for any $1 \leq n \leq N$\,, the weak solution $(c,\mu,\vec{v},p)$ to model problem \eqref{eq:CHNS:model} satisfies
\begin{subequations}\label{eq:CHNS:consistency}
\begin{eqnarray}
\big({\partial_t c}(t^n),\chi\big) + a_{\mathcal{D}}(\mu^n,\chi) + a_{\mathcal{A}}(c^n,\vec{v}^n,\chi) = 0, \quad \forall \chi \in S_h,\\
\big(\Phi_{+}\,\!'(c^n)+\Phi_{-}\,\!'(c^n),\varphi\big) + \kappa a_{\mathcal{D}}(c^n,\varphi) - (\mu^n,\varphi) = 0, \quad \forall \varphi \in S_h,\\
\begin{aligned}
\big({\partial_t \vec{v}}(t^n),\vec{\theta}\big) + a_{\mathcal{C}}(\vec{v}^n,\vec{v}^n,\vec{v}^n,\vec{\theta}) + \mu_\mathrm{s} a_\strain(\vec{v}^n, \vec{\theta}) \hspace*{7.85em} \\+ b_{\mathcal{P}}(p^n,\vec{\theta}) - b_\mathcal{I}(c^n,\mu^n,\vec{\theta}) = 0, \quad \forall \vec{\theta} \in \mathbf{X}_h,
\end{aligned}\\
b_{\mathcal{P}}(\phi,\vec{v}^n) = 0, \quad \forall \phi \in Q_h.
\end{eqnarray}
\end{subequations}
Before starting error analysis, let us briefly review several useful definitions and properties. 
Let $\Pi_h: L^2(\Omega) \rightarrow Q_h$ be the L$^2$ projection operator onto $Q_h$:
\begin{align*}
(\Pi_h{\omega}-\omega,\phi) = 0, && \forall \phi \in Q_h, \quad \forall\omega\in L^2(\Omega).
\end{align*}

Lax--Milgram theorem allows us to define an invertible operator $\mathcal{J}: M_h \rightarrow M_h$ via the following variational problem: given $\lambda \in M_h$\,, for any $\phi \in M_h$\,, find $\mathcal{J}(\lambda) \in M_h$ such that
\begin{equation}\label{eq:CHNS:define_operator_J}
a_{\mathcal{D}}(\phi,\mathcal{J}(\lambda)) = (\lambda,\phi).
\end{equation}
\begin{lemma}\label{lem:CHNS:error_property_of_J}
The operator $\mathcal{J}$ is linear and the identity \eqref{eq:CHNS:define_operator_J} still holds for any $\phi \in S_h$ and any $\lambda \in M_h$.  In addition, there exists a constant $C_1>0$ independent of mesh size $h$, such that 
\begin{align*}
\abs{(\lambda, \phi)} \leq C_1 \norm{\phi}{\mathrm{DG}} \norm{\mathcal{J}(\lambda)}{\mathrm{DG}}, &&
\forall \phi \in H^1(\setE_h)\,,\quad \forall \lambda \in M_h.
\end{align*} 
\end{lemma}
\begin{proof}
The linearity of the operator $\mathcal{J}$ is easy to check. 
For any $\phi \in S_h$ and any $\lambda \in M_h$\,, due to the fact $\phi - \frac{1}{\abs{\Omega}}\int_\Omega\phi$ belongs to $M_h$\,, we have
\begin{align*}
a_{\mathcal{D}}\big(\phi,\mathcal{J}(\lambda)\big)
&= a_{\mathcal{D}}\big(\phi - \frac{1}{\abs{\Omega}}\int_\Omega\phi, \mathcal{J}(\lambda)\big) + a_{\mathcal{D}}\big(\frac{1}{\abs{\Omega}}\int_\Omega\phi, \mathcal{J}(\lambda)\big) \\
&= (\lambda, \phi - \frac{1}{\abs{\Omega}}\int_\Omega\phi)
= (\lambda, \phi) - (\lambda, \frac{1}{\abs{\Omega}}\int_\Omega\phi)
= (\lambda, \phi).
\end{align*}
Let $\tilde{\Pi}_h: H^1(\setE_h) \rightarrow S_h$ denote the $L^2$ projection operator onto $S_h$.
It is easy to show that $\tilde{\Pi}_h$ is stable with respect to the DG norm, i.\,e., we have the inequality $\norm{\tilde{\Pi}_h \phi}{\mathrm{DG}} \leq C \norm{\phi}{\mathrm{DG}}$\,. Therefore, by triangular inequality, the definition of operator $\mathcal{J}$, and the continuity of $a_{\mathcal{D}}$, we obtain for any $\lambda$ in $M_h$\,:
\begin{align*}
\abs{(\lambda, \phi)} 
&\leq \abs{(\phi-\tilde{\Pi}_h\phi, \lambda)} + \abs{(\tilde{\Pi}_h\phi, \lambda)} 
= \abs{(\tilde{\Pi}_h\phi, \lambda)}\\
&= a_{\mathcal{D}}(\tilde{\Pi}_h \phi,\mathcal{J}(\lambda)) 
\leq C_\alpha \norm{\tilde{\Pi}_h\phi}{\mathrm{DG}}\norm{\mathcal{J}(\lambda)}{\mathrm{DG}},
\end{align*}
which concludes our proof.
\end{proof}
We recall the following approximation operator (see Lemma~6.1 in \cite{ChaabaneGiraultPuelzRiviere2017}).
\begin{lemma}
\label{lem:Rh}
There is an approximation operator $\mathcal{R}_h: H_0^1(\setE_h)^d \rightarrow \mathbf{X}_h$
satisfying
\begin{equation}\label{eq:approxRh1}
b_\mathcal{P}(\phi,\mathcal{R}_h(\vec{v})-\vec{v}) = 0,\quad \forall\vec{v}\in H_0^1(\setE_h)^d, \quad \forall \phi\in Q_h,
\end{equation}
and for all $E$ in $\setE_h$, for all $\vec{v}$ in $H_0^1(\setE_h)^d \cap W^{s,r}(E)^d$, $1 \leq r \leq \infty$, $1 \leq s \leq q+1$,
\begin{align}\label{eq:approxRh2}
\begin{split}
\norm{\mathcal{R}_h(\vec{v})-\vec{v}}{L^r(E)} &\leq C h^s |\vec{v}|_{W^{s,r}(\Delta_E)},\\
\norm{\nabla(\mathcal{R}_h(\vec{v})-\vec{v})}{L^r(E)} &\leq C h^{s-1} \vert\vec{v}\vert_{W^{s,r}(\Delta_E)}, 
\end{split}
\end{align}
with a constant $C$ independent of mesh size $h$ and $E$, where $\Delta_E \subset \Omega$ is a macro-element. We also have for all $s$, $1 \leq s \leq q+1$,
\begin{equation}
\label{eq:approxRh3}
\forall\vec{v}\in H_0^1(\setE_h)^d \cap H^s(\Omega)^d,\quad \Vert \mathcal{R}_h(\vec{v})-\vec{v}\Vert_{\mathrm{DG}} \leq C h^{s-1} \vert\vec{v}\vert_{H^s(\Omega)}.
\end{equation}
\end{lemma}
With the operator $\mathcal{R}_h$, we have a bound for the form $a_\mathcal{C}$ (see Proposition 6.2 in \cite{ChaabaneGiraultPuelzRiviere2017}).
\begin{lemma}[Bounds of $a_{\mathcal{C}}$]\label{lem:CHNS:boundconvection2}
There exists a constant $C$ independent of mesh size $h$ such that for any $\vec{u}$ in $(L^\infty(\Omega)\cap W^{1,3}(\Omega)\cap H^{3/2}(\Omega))^d$,
any $\vec{v}_h$ in $\mathbf{V}_h$ and any $\vec{w}_h, \vec{z}_h$ in $\vec{X}_h$, the bound holds
\begin{multline*}
|a_\mathcal{C}(\vec{z}_h,\vec{v}_h,\vec{u}-\mathcal{R}_h\vec{u},\vec{w}_h)| \leq C \left( \Vert \vec{u}-\mathcal{R}_h\vec{u}\Vert_{L^\infty(\Omega)}
+ \vert \vec{u}-\mathcal{R}_h\vec{u}\vert_{W^{1,3}(\Omega)} + \vert \vec{u}\vert_{H^{3/2}(\Omega)}\right) \\
\times \Vert \vec{v}_h\Vert_{L^2(\Omega)} \Vert \vec{w}_h\Vert_{\mathrm{DG}}.
\end{multline*}
\end{lemma}
Recall that $\mathcal{P}_h c^n$ and $\mathcal{P}_h \mu^n$ are the elliptic projections of $c^n$ and $\mu^n$, 
which are defined in \eqref{eq:CHNS:elliptic_proj_c}.  The DG error analysis for elliptic problems yields the following error bounds \cite{riviere2008}.
\begin{lemma}\label{lem:CHNS:error_elliptic_projection}
There exist a constant $C$, independent of mesh size $h$ and time step size $\tau$, such that for all $0 \leq n \leq N$
\begin{align*}
\norm{c^n - \mathcal{P}_hc^n}{\mathrm{DG}} &\leq Ch^q \norm{c}{L^\infty(0,T;\,H^{q+1}(\Omega))},\\ 
\norm{\mu^n - \mathcal{P}_h\mu^n}{\mathrm{DG}} &\leq Ch^q \norm{\mu}{L^\infty(0,T;\,H^{q+1}(\Omega))},\\
\norm{\delta_\tau (c^n - \mathcal{P}_hc^n)}{L^2(\Omega)} &\leq C h^q \norm{\partial_t c}{L^\infty(0,T;\,H^{q+1}(\Omega))}.
\end{align*}
\end{lemma}
We define the projection errors and the discretization errors as follows
\begin{align*}
\zeta_c^n  &= c^n - \mathcal{P}_hc^n, & 
\xi_c^n &= \mathcal{P}_hc^n - c_h^n,\\ 
\zeta_\mu^n &= \mu^n - \mathcal{P}_h\mu^n, & 
\xi_\mu^n &= \mathcal{P}_h\mu^n - \mu_h^n,\\ 
\bfzeta_{\vec{v}}^n  &= \vec{v}^n - \mathcal{R}_h\vec{v}^n, & 
\bfxi_{\vec{v}}^n &= \mathcal{R}_h\vec{v}^n - \vec{v}_h^n,\\ 
\zeta_p^n &= p^n - \Pi_h p^n, &
\xi_p^n &= \Pi_h p^n - p_h^n. 
\end{align*}
Now we are in the position of stating the error equation. We note that for all $n \geq 1$
\begin{align*}
a_{\mathcal{D}}(\zeta_\mu^n,\chi)=0, \quad b_{\mathcal{P}}(\phi,\bfzeta_{\vec{v}}^n)=0, && \forall \chi \in S_h, \quad \forall \phi \in Q_h.
\end{align*}
Therefore, from \eqref{eq:CHNS:DGscheme} and \eqref{eq:CHNS:consistency}, the error equation becomes, for any $\chi \in S_h$\,, $\varphi \in S_h$\,, $\vec{\theta} \in \mathbf{X}_h$\,, and $\phi \in Q_h$\,:
\begin{subequations}
\begin{eqnarray}
\begin{aligned}\label{eq:CHNS:error_equation1}
(\delta_\tau \xi_c^n,\chi) + a_{\mathcal{D}}(\xi_\mu^n,\chi) = \big(\delta_\tau c^n -(\partial_t c)^n - \delta_\tau\zeta_c^n,\chi\big) \hspace*{7.5em}
\\
- a_{\mathcal{A}}(c^n,\vec{v}^n,\chi) + a_{\mathcal{A}}(c_h^{n-1},\vec{v}_h^n,\chi), 
\end{aligned}\\
\begin{aligned}\label{eq:CHNS:error_equation2}
\kappa a_{\mathcal{D}}(\xi_c^n,\varphi) - (\xi_\mu^n,\varphi) = (\zeta_\mu^n,\varphi) \hspace*{15.9em}\\ 
+ \big(\Phi_+\,\!'(c_h^n)-\Phi_+\,\!'(c^n),\varphi\big) + \big(\Phi_-\,\!'(c_h^{n-1})-\Phi_-\,\!'(c^{n}),\varphi\big), 
\end{aligned}\\
\begin{aligned}\label{eq:CHNS:error_equation3}
(\delta_\tau\bfxi_{\vec{v}}^n,\vec{\theta}) 
+ \mu_\mathrm{s} a_\strain(\bfxi_{\vec{v}}^n, \vec{\theta}) 
+ b_{\mathcal{P}}(\xi_p^n,\vec{\theta}) 
= \big(\delta_\tau\vec{v}^n-(\partial_t \vec{v})^n-\delta_\tau\bfzeta_{\vec{v}}^n ,\vec{\theta}\big) \hspace*{1.275em}\\ 
- \mu_\mathrm{s} a_\strain(\bfzeta_{\vec{v}}^n, \vec{\theta}) 
- b_{\mathcal{P}}(\zeta_p^n,\vec{\theta}) 
- a_{\mathcal{C}}(\vec{v}^n,\vec{v}^n,\vec{v}^n,\vec{\theta}) \hspace*{7.5em}\\
+ a_{\mathcal{C}}(\vec{v}_h^{n-1},\vec{v}_h^{n-1},\vec{v}_h^n,\vec{\theta})
+ b_\mathcal{I}(c^n,\mu^n,\vec{\theta}) - b_\mathcal{I}(c_h^{n-1},\mu_h^{n},\vec{\theta}), 
\end{aligned}\\
b_{\mathcal{P}}(\phi,\bfxi_{\vec{v}}^n) = 0. \label{eq:CHNS:error_equation4}
\end{eqnarray}
\end{subequations}
%
%
%
We now state the main theorem. 
\begin{theorem}
Suppose $(c,\mu,\vec{v},p)$ is a weak solution of \eqref{eq:CHNS:consistency} with regularity \eqref{eq:CHNS:error_regularities}. Then, under \cref{as:CHNS:assumptionA} and sufficiently small time step size $\tau$, there exists a constant $C$ independent of mesh size $h$ and time step size $\tau$ such that for any $m \geq 1$
\begin{align*}
\max_{1\leq n\leq m}\Big(\norm{\xi_c^n }{\mathrm{DG}}^2 + \norm{\bfxi_{\vec{v}}^n}{L^2(\Omega)}^2\Big) 
+ \tau\sum_{n=1}^{m}\norm{\bfxi_{\vec{v}}^n}{\mathrm{DG}}^2 &\leq C(\tau^2 + h^{2q}),\\
\tau \sum_{n=1}^m \norm{\xi_\mu^n}{\mathrm{DG}}^2 &\leq C (\tau^2 + h^{2q}).
\end{align*}
\end{theorem}
\begin{proof}
From \cref{thm:CHNS:discrete_mass_conservation} and \eqref{eq:CHNS:elliptic_proj_c}, 
it is obvious that $\delta_\tau \xi_c^n$ belongs to $M_h$, which means that the function $\mathcal{J}(\delta_\tau \xi_c^n)$ 
is well defined in $M_h$. Choosing $\chi = \mathcal{J}(\delta_\tau \xi_c^n)$ in \eqref{eq:CHNS:error_equation1} and using \cref{lem:CHNS:error_property_of_J}, we have
\begin{subequations}
\begin{multline}\label{eq:CHNS:error_analysis1}
\begin{aligned}
a_{\mathcal{D}}\big(\mathcal{J}(\delta_\tau\xi_c^n),\mathcal{J}(\delta_\tau \xi_c^n)\big) + (\delta_\tau \xi_c^n, \xi_\mu^n)
= \big(\delta_\tau c^n -(\partial_t c)^n - \delta_\tau\zeta_c^n,\mathcal{J}(\delta_\tau \xi_c^n)\big)\\ 
- a_{\mathcal{A}}\big(c^n,\vec{v}^n,\mathcal{J}(\delta_\tau \xi_c^n)\big) 
+ a_{\mathcal{A}}\big(c_h^{n-1},\vec{v}_h^n,\mathcal{J}(\delta_\tau \xi_c^n)\big).
\end{aligned}
\end{multline}
Choosing $\varphi = \delta_\tau \xi_c^n$ in \eqref{eq:CHNS:error_equation2} and adding and subtracting the appropriate terms, we obtain
\begin{multline}\label{eq:CHNS:error_analysis2}
\begin{aligned}
\kappa a_{\mathcal{D}}(\xi_c^n,\delta_\tau \xi_c^n) - (\xi_\mu^n,\delta_\tau \xi_c^n) 
= (\zeta_\mu^n,\delta_\tau \xi_c^n) + \big(\Phi_+\,\!'(c_h^n)-\Phi_+\,\!'(c^n),\delta_\tau \xi_c^n\big)\\ 
+ \big(\Phi_-\,\!'(c_h^{n-1})-\Phi_-\,\!'(c^{n-1}),\delta_\tau \xi_c^n\big) 
+ \big(\Phi_-\,\!'(c^{n-1})-\Phi_-\,\!'(c^{n}),\delta_\tau \xi_c^n\big).
\end{aligned}
\end{multline}
Choosing $\vec{\theta} = \bfxi_{\vec{v}}^n$ in \eqref{eq:CHNS:error_equation3}, $\phi = -\xi_p^n$ in \eqref{eq:CHNS:error_equation4} and combining the resulting equations, we have
\begin{multline}\label{eq:CHNS:error_analysis3}
\begin{aligned}
(\delta_\tau\bfxi_{\vec{v}}^n,\bfxi_{\vec{v}}^n) 
+ \mu_\mathrm{s} a_\strain(\bfxi_{\vec{v}}^n, \bfxi_{\vec{v}}^n) 
= \big(\delta_\tau\vec{v}^n-(\partial_t \vec{v})^n-\delta_\tau\bfzeta_{\vec{v}}^n,\bfxi_{\vec{v}}^n\big) \\
- \mu_\mathrm{s} a_\strain(\bfzeta_{\vec{v}}^n, \bfxi_{\vec{v}}^n) 
- b_{\mathcal{P}}(\zeta_p^n,\bfxi_{\vec{v}}^n) 
- a_{\mathcal{C}}(\vec{v}^n,\vec{v}^n,\vec{v}^n,\bfxi_{\vec{v}}^n) \\+ a_{\mathcal{C}}(\vec{v}_h^{n-1},\vec{v}_h^{n-1},\vec{v}_h^n,\bfxi_{\vec{v}}^n) 
+ b_\mathcal{I}(c^n,\mu^n,\bfxi_{\vec{v}}^n) - b_\mathcal{I}(c_h^{n-1},\mu_h^{n},\bfxi_{\vec{v}}^n).
\end{aligned}
\end{multline}
\end{subequations}
Summing \eqref{eq:CHNS:error_analysis1} -- \eqref{eq:CHNS:error_analysis3}, we obtain the following equation
\begin{eqnarray}\label{eq:error:error_equation_major}
\begin{aligned}
a_{\mathcal{D}}\big(\mathcal{J}(\delta_\tau\xi_c^n),\mathcal{J}(\delta_\tau \xi_c^n)\big) + \kappa a_{\mathcal{D}}(\xi_c^n,\delta_\tau \xi_c^n) + \mu_\mathrm{s} a_\strain(\bfxi_{\vec{v}}^n, \bfxi_{\vec{v}}^n) + (\delta_\tau\bfxi_{\vec{v}}^n,\bfxi_{\vec{v}}^n)\\
= \big(\delta_\tau c^n -(\partial_t c)^n,\mathcal{J}(\delta_\tau \xi_c^n)\big) - \big(\delta_\tau\zeta_c^n,\mathcal{J}(\delta_\tau \xi_c^n)\big) 
+ \big(\delta_\tau\vec{v}^n-(\partial_t \vec{v})^n,\bfxi_{\vec{v}}^n\big) \\
- (\delta_\tau\bfzeta_{\vec{v}}^n,\bfxi_{\vec{v}}^n) + \big(\Phi_+\,\!'(c_h^n)-\Phi_+\,\!'(c^n),\delta_\tau \xi_c^n\big) 
+ \big(\Phi_-\,\!'(c_h^{n-1})-\Phi_-\,\!'(c^{n-1}),\delta_\tau \xi_c^n\big) \\
+ \big(\Phi_-\,\!'(c^{n-1})-\Phi_-\,\!'(c^{n}),\delta_\tau \xi_c^n\big) + (\zeta_\mu^n,\delta_\tau \xi_c^n) - \mu_\mathrm{s} a_\strain(\bfzeta_{\vec{v}}^n, \bfxi_{\vec{v}}^n) - b_{\mathcal{P}}(\zeta_p^n,\bfxi_{\vec{v}}^n) \\
- a_{\mathcal{C}}(\vec{v}^n,\vec{v}^n,\vec{v}^n,\bfxi_{\vec{v}}^n) 
+ a_{\mathcal{C}}(\vec{v}_h^{n-1},\vec{v}_h^{n-1},\vec{v}_h^n,\bfxi_{\vec{v}}^n) 
- a_{\mathcal{A}}\big(c^n,\vec{v}^n,\mathcal{J}(\delta_\tau \xi_c^n)\big) \\
+ a_{\mathcal{A}}\big(c_h^{n-1},\vec{v}_h^n,\mathcal{J}(\delta_\tau \xi_c^n)\big) 
+ b_\mathcal{I}(c^n,\mu^n,\bfxi_{\vec{v}}^n) 
- b_\mathcal{I}(c_h^{n-1},\mu_h^{n},\bfxi_{\vec{v}}^n)
= T_1 + \dots + T_{16}.
\end{aligned}
\end{eqnarray}
The remainder of the proof consists of finding lower bounds for the terms in the left-hand side and upper bounds for the terms in the right-hand side of the equation above. We will then utilize Gronwall's lemma.
For the left-hand side of \eqref{eq:error:error_equation_major}, since $a_{\mathcal{D}}$ and the inner product are both symmetric bilinear forms, using the formula $a(a-b) \geq \frac{1}{2}a^2 - \frac{1}{2}b^2$, and the coercivity of $a_{\mathcal{D}}$ and $a_\strain$, we have
\begin{align}\label{eq:CHNS:bound_error_lefthand}
\begin{split}
&a_{\mathcal{D}}\big(\mathcal{J}(\delta_\tau\xi_c^n),\mathcal{J}(\delta_\tau \xi_c^n)\big) 
+ \kappa a_{\mathcal{D}}(\xi_c^n,\delta_\tau \xi_c^n) 
+ \mu_\mathrm{s} a_\strain(\bfxi_{\vec{v}}^n, \bfxi_{\vec{v}}^n) 
+ (\delta_\tau\bfxi_{\vec{v}}^n,\bfxi_{\vec{v}}^n)\\
\geq& K_\alpha\norm{\mathcal{J}(\delta_\tau\xi_c^n)}{\mathrm{DG}}^2 
+ \frac{\kappa}{2\tau}a_{\mathcal{D}}(\xi_c^n ,\xi_c^n) 
- \frac{\kappa}{2\tau}a_{\mathcal{D}}(\xi_c^{n-1},\xi_c^{n-1})\\
&+ \mu_\mathrm{s}K_\strain\norm{\bfxi_{\vec{v}}^n}{\mathrm{DG}}^2 +
\frac{1}{2\tau}\norm{\bfxi_{\vec{v}}^n}{L^2(\Omega)}^2 - \frac{1}{2\tau}\norm{\bfxi_{\vec{v}}^{n-1}}{L^2(\Omega)}^2.
\end{split}
\end{align}
Now, let us proceed to estimate the right-hand side of \eqref{eq:error:error_equation_major} term by term. 
At several places, we will use Young's inequality with positive real numbers $r_1$, $r_2$, and $r_3$ to be chosen later.
By Cauchy--Schwarz's inequality, Poincar\'e's inequality, Young's inequality, and using a Taylor expansion, we have 
\begin{align*}
T_1 &\leq \norm{\delta_\tau c^n - (\partial_t c)^n}{L^2(\Omega)}\norm{\mathcal{J}(\delta_\tau \xi_c^n)}{L^2(\Omega)} \\
&\leq C_P \norm{\delta_\tau c^n - (\partial_t c)^n}{L^2(\Omega)}\norm{\mathcal{J}(\delta_\tau \xi_c^n)}{\mathrm{DG}}\\
&\leq \frac{r_1}{2}\norm{\mathcal{J}(\delta_\tau \xi_c^n)}{\mathrm{DG}}^2 
+ \frac{C_P^2}{2r_1} \norm{\delta_\tau c^n - (\partial_t c)^n}{L^2(\Omega)}^2 \\
&\leq \frac{r_1}{2} \norm{\mathcal{J}(\delta_\tau \xi_c^n)}{\mathrm{DG}}^2 
+ \frac{C_P^2\tau}{6r_1}\int_{t^{n-1}}^{t^n} \norm{\partial_{tt} c}{L^2(\Omega)}^2.
\end{align*}
By Cauchy--Schwarz's inequality, Poincar\'e's inequality, and Young's inequality, the term $T_2$ is simply bounded 
\begin{align*}
T_2 &\leq \norm{\delta_\tau \zeta_c^n}{L^2(\Omega)}\norm{\mathcal{J}(\delta_\tau \xi_c^n)}{L^2(\Omega)}\\
&\leq C_P\norm{\delta_\tau \zeta_c^n}{L^2(\Omega)}\norm{\mathcal{J}(\delta_\tau \xi_c^n)}{\mathrm{DG}}\\
&\leq \frac{r_1}{2}\norm{\mathcal{J}(\delta_\tau \xi_c^n)}{\mathrm{DG}}^2 
+ \frac{C_P^2}{2 r_1}\norm{\delta_\tau \zeta_c^n}{L^2(\Omega)}^2.
\end{align*}
Next, we remark that $\mathcal{P}_h (\delta_\tau c^n) = \delta_\tau (\mathcal{P}_h c^n)$ and 
with the approximation result of \cref{lem:CHNS:error_elliptic_projection}, we have
\[
\Vert \delta_\tau\zeta_c^n\Vert_{L^2(\Omega)} \leq C h^q \Vert \partial_t c\Vert_{L^\infty(0,T;H^{q+1}(\Omega))}.
\]
Therefore, we have
\[
T_2 \leq \frac{r_1}{2}\norm{\mathcal{J}(\delta_\tau \xi_c^n)}{\mathrm{DG}}^2 + \frac{C}{r_1} h^{2q}.
\]
The terms $T_3$ and $T_4$ are bounded by employing a similar technique as for $T_1$ and $T_2$. 
\begin{align*}
T_3 &\leq \norm{\bfxi_{\vec{v}}^n}{L^2(\Omega)}^2 + \frac{\tau}{12}\int_{t^{n-1}}^{t^n} \norm{\partial_{tt} \vec{v}}{L^2(\Omega)}^2,\\
T_4 &\leq \norm{\bfxi_{\vec{v}}^n}{L^2(\Omega)}^2 + \frac{1}{4}\norm{\delta_\tau \bfzeta_{\vec{v}}^n}{L^2(\Omega)}^2.
\end{align*}
We write
\[
\delta_\tau \bfzeta_{\vec{v}}^n = \frac{1}{\tau}\int_{t^{n-1}}^{t^n} \partial_t (\vec{v}-\mathcal{R}_h\vec{v}).
\]
This implies with \cref{lem:Rh}:
\begin{equation}\label{eq:bounddeltatau}
\Vert \delta_\tau \bfzeta_{\vec{v}}^n\Vert_{L^2(\Omega)} \leq \frac{C}{\sqrt{\tau}} h^{q} \Vert \partial_t \vec{v}\Vert_{L^2(t^{n-1},t^n;H^q(\Omega))}.
\end{equation}
Therefore we have for $T_4$:
\[
T_4 \leq \norm{\bfxi_{\vec{v}}^n}{L^2(\Omega)}^2 + \frac{C}{\tau} h^{2q} \Vert \partial_t \vec{v}\Vert_{L^2(t^{n-1},t^n;H^q(\Omega))}^2.
\]
For the term $T_5$, using \Cref{lem:CHNS:error_property_of_J}, Young's inequality, assumption \eqref{eq:CHNS:error_assumption_on_Phi}, and triangular inequality, we have
\begin{align*}
T_5 &\leq C_1 \norm{\Phi_+\,\!'(c_h^n)-\Phi_+\,\!'(c^n)}{\mathrm{DG}}\norm{\mathcal{J}(\delta_\tau \xi_c^n)}{\mathrm{DG}}\\
&\leq \frac{r_1}{2} \norm{\mathcal{J}(\delta_\tau \xi_c^n)}{\mathrm{DG}}^2 + \frac{C_1^2}{2 r_1} \norm{\Phi_+\,\!'(c_h^n)-\Phi_+\,\!'(c^n)}{\mathrm{DG}}^2\\
&\leq \frac{r_1}{2} \norm{\mathcal{J}(\delta_\tau \xi_c^n)}{\mathrm{DG}}^2 + \frac{C_1^2 C_\mathrm{lip}^2}{2 r_1} \norm{c_h^n-c^n}{\mathrm{DG}}^2\\
&\leq \frac{r_1}{2} \norm{\mathcal{J}(\delta_\tau \xi_c^n)}{\mathrm{DG}}^2 + \frac{C_1^2 C_\mathrm{lip}^2}{r_1} (\norm{\xi_c^n}{\mathrm{DG}}^2 
+ C h^{2q} \norm{c}{L^\infty(0,T;H^{q+1}(\Omega))}^2).
\end{align*}
Repeating exactly the same steps of bounding the term $T_5$ as above, we have the following bound for the term $T_6$
\begin{equation*}
T_6 \leq \frac{r_1}{2} \norm{\mathcal{J}(\delta_\tau \xi_c^n)}{\mathrm{DG}}^2 
+ \frac{C_1^2 C_\mathrm{lip}^2}{r_1} (\norm{\xi_c^{n-1}}{\mathrm{DG}}^2 + C h^{2q} \norm{c}{L^\infty(0,T;H^{q+1}(\Omega))}^2).
\end{equation*}
For the term $T_7$, we use \cref{lem:CHNS:error_property_of_J}, Young's inequality, and a Taylor expansion of first order and obtain
\begin{align*}
T_7 &\leq \frac{r_1}{2} \norm{\mathcal{J}(\delta_\tau \xi_c^n)}{\mathrm{DG}}^2
+ \frac{C_1^2}{2r_1} \norm{\Phi_-\,\!'(c^{n-1})-\Phi_-\,\!'(c^{n})}{\mathrm{DG}}^2\\
&\leq \frac{r_1}{2} \norm{\mathcal{J}(\delta_\tau \xi_c^n)}{\mathrm{DG}}^2
+ \frac{C\tau^2}{r_1} \norm{\partial_t c}{L^\infty(0,T;\,H^1(\Omega))}^2.
\end{align*}
For the term $T_8$, since $\zeta_\mu^n$ belongs to $H^1(\setE_h)$, with \Cref{lem:CHNS:error_property_of_J}, Young's inequality 
and the approximation bound in \cref{lem:CHNS:error_elliptic_projection},  we have
\begin{equation*}
T_8 \leq C_1 \norm{\zeta_\mu^n}{\mathrm{DG}} \norm{\mathcal{J}(\delta_\tau \xi_c^n)}{\mathrm{DG}}
\leq \frac{r_1}{2} \norm{\mathcal{J}(\delta_\tau \xi_c^n)}{\mathrm{DG}}^2 + \frac{C h^{2q}}{r_1} \norm{\mu}{L^\infty(0,T;H^{q+1}(\Omega))}^2.
\end{equation*}
The way of processing the terms $T_9$ and $T_{10}$ follows the argument in \cite{riviere2008} (page 127), we have
\begin{align*}
T_9 &\leq r_2 \mu_s K_\strain \norm{\bfxi_{\vec{v}}^n}{\mathrm{DG}}^2 + \frac{C}{r_2}h^{2q}\norm{\vec{v}^n}{H^{q+1}(\Omega)}^2,\\
T_{10} &\leq r_2 \mu_s K_\strain \norm{\bfxi_{\vec{v}}^n}{\mathrm{DG}}^2 + \frac{C}{r_2} h^{2q}\norm{p^n}{H^{q}(\Omega)}^2. 
\end{align*}
The handling of the terms $T_{11}$ and $T_{12}$ is complicated; however these terms have been analyzed in papers for the Navier--Stokes equations,
for instance in \cite{ChaabaneGiraultPuelzRiviere2017}.  We give an outline of the proof for completeness.
\begin{align}\label{eq:CHNS:error_convection}
\begin{split}
T_{11} + T_{12} 
=& - a_{\mathcal{C}}(\vec{v}_h^{n-1},\vec{v}^n,\vec{v}^n,\bfxi_{\vec{v}}^n) 
+ a_{\mathcal{C}}(\vec{v}_h^{n-1},\vec{v}_h^{n-1},\vec{v}_h^n,\bfxi_{\vec{v}}^n)\\
=& -a_{\mathcal{C}}(\vec{v}_h^{n-1},\vec{v}_h^{n-1},\bfxi_{\vec{v}}^n,\bfxi_{\vec{v}}^n) 
- a_{\mathcal{C}}(\vec{v}_h^{n-1},\bfxi_{\vec{v}}^{n-1},\mathcal{R}_h\vec{v}^n,\bfxi_{\vec{v}}^n)\\ 
&- a_{\mathcal{C}}\big(\vec{v}_h^{n-1},\bfzeta_{\vec{v}}^{n-1},\mathcal{R}_h\vec{v}^n,\bfxi_{\vec{v}}^n\big) 
+ a_{\mathcal{C}}\big(\vec{v}_h^{n-1},\vec{v}^{n-1}-\vec{v}^n,\mathcal{R}_h\vec{v}^n,\bfxi_{\vec{v}}^n\big)\\
&- a_{\mathcal{C}}\big(\vec{v}_h^{n-1},\vec{v}^n,\bfzeta_{\vec{v}}^n,\bfxi_{\vec{v}}^n\big) 
\\
=& T_{\mathcal{C}}^1 + \dots + T_{\mathcal{C}}^5.
\end{split}
\end{align}
We know from \cref{lem:CHNS:positivity_convection} that the first term $T_{\mathcal{C}}^1$ is negative. 
We rewrite the second term as:
\[
T_{\mathcal{C}}^2 = -a_{\mathcal{C}}(\vec{v}_h^{n-1},\bfxi_{\vec{v}}^{n-1},\vec{v}^n, \bfxi_{\vec{v}}^n) 
+ a_{\mathcal{C}}(\vec{v}_h^{n-1},\bfxi_{\vec{v}}^{n-1},\bfzeta_{\vec{v}}^n, \bfxi_{\vec{v}}^n).
\]
Note that $\bfxi_{\vec{v}}^{n-1}$ belongs to $V_h$ and we apply \cref{lem:CHNS:boundconvection} to the first term
\[
\abs{a_{\mathcal{C}}(\vec{v}_h^{n-1},\bfxi_{\vec{v}}^{n-1},\vec{v}^n, \bfxi_{\vec{v}}^n)}
\leq C \left( \Vert \vec{v}^n\Vert_{L^\infty(\Omega)} 
+ \vert \vec{v}^n\vert_{W^{1,3}(\Omega)}\right) \Vert \bfxi_{\vec{v}}^{n-1}\Vert_{L^2(\Omega)}
\Vert \bfxi_{\vec{v}}^n \Vert_{\mathrm{DG}}.
\]
We apply \cref{lem:CHNS:boundconvection2} to the second term
\begin{multline*}
\abs{a_{\mathcal{C}}(\vec{v}_h^{n-1},\bfxi_{\vec{v}}^{n-1},\bfzeta_{\vec{v}}^n, \bfxi_{\vec{v}}^n)}
\leq C \left( \Vert \bfzeta_{\vec{v}}^n\Vert_{L^\infty(\Omega)} 
+ \vert \bfzeta_{\vec{v}}^n\vert_{W^{1,3}(\Omega)}+\vert \vec{v}^n\vert_{H^{3/2}(\Omega)}\right) \\
\times \Vert \bfxi_{\vec{v}}^{n-1}\Vert_{L^2(\Omega)}
\Vert \bfxi_{\vec{v}}^n \Vert_{\mathrm{DG}}.
\end{multline*}
Combining both terms, we obtain
\[
T_{\mathcal{C}}^2 \leq r_2 \mu_s K_\strain \Vert \bfxi_{\vec{v}}^n\Vert_{\mathrm{DG}}^2
+ \frac{C}{r_2} \Vert \bfxi_{\vec{v}}^{n-1}\Vert_{L^2(\Omega)}^2.
\]
We apply \cref{lem:CHNS:continuityaC} to the term $T_\mathcal{C}^3$ and obtain using the regularity of the weak solution:
\[
T_\mathcal{C}^3 \leq r_2 \mu_s K_\strain \Vert \bfxi_{\vec{v}}^n\Vert_{\mathrm{DG}}^2
+ \frac{C}{r_2} h^{2q} \Vert \vec{v}\Vert_{L^\infty(0,T;H^{q+1}(\Omega))}^2.
\]
For the term $T_\mathcal{C}^4$, we note that $\vec{v}^n$ is divergence free and has no jumps. The term simplifies
\begin{multline*}
T_\mathcal{C}^4 = \sum_{E\in\setE_h} \int_E \left((\vec{v}^{n-1}-\vec{v}^n)\cdot\nabla \mathcal{R}_h\vec{v}^n\right) \cdot \bfxi_{\vec{v}}^n
\\
+ \sum_{E\in\setE_h} \int_{\partial E_-\setminus \partial\Omega} \vert (\vec{v}^{n-1}-\vec{v}^n)\cdot\vec{n}_E\vert
\left((\mathcal{R}_h\vec{v}^n)^{\mathrm{int}}-(\mathcal{R}_h\vec{v}^n)^{\mathrm{ext}}\right)\cdot(\bfxi_{\vec{v}}^n)^{\mathrm{int}}.
\end{multline*}
We then have
\[
|\int_E \left((\vec{v}^{n-1}-\vec{v}^n)\cdot\nabla \mathcal{R}_h\vec{v}^n\right) \cdot \bfxi_{\vec{v}}^n \vert
\leq \int_{t^{n-1}}^{t^n} \Vert \partial_t \vec{v}\Vert_{L^2(E)} \Vert \bfxi_{\vec{v}}^n \Vert_{L^6(E)} \vert \mathcal{R}_h\vec{v}^n\vert_{W^{1,3}(E)}.
\]
Using the stability of $\mathcal{R}_h$ and summing over the elements yields:
\begin{multline*}
\left| \sum_{E\in\setE_h} \int_E \left((\vec{v}^{n-1}-\vec{v}^n)\cdot\nabla \mathcal{R}_h\vec{v}^n\right) \cdot \bfxi_{\vec{v}}^n\right|
\\
\leq C \sqrt{\tau} \Vert \partial_t \vec{v}\Vert_{L^2(t^{n-1},t^n;L^2(\Omega))} \Vert \bfxi_{\vec{v}}^n\Vert_{\mathrm{DG}}
\Vert \vec{v}\Vert_{L^\infty(0,T;W^{1,3}(\Omega))}.
\end{multline*}
For the term on the faces, we rewrite
\[
\vert (\mathcal{R}_h\vec{v}^n)^{\mathrm{int}}-(\mathcal{R}_h\vec{v}^n)^{\mathrm{ext}}\vert  = 
\vert [\mathcal{R}_h\vec{v}^n - \vec{v}^n]\vert,
\]
and employ trace inequalities to obtain a similar bound:
\begin{multline*}
\left|\sum_{E\in\setE_h} \int_{\partial E_-\setminus \partial\Omega} \vert (\vec{v}^{n-1}-\vec{v}^n)\cdot\vec{n}_E\vert
\left((\mathcal{R}_h\vec{v}^n)^{\mathrm{int}}-(\mathcal{R}_h\vec{v}^n)^{\mathrm{ext}}\right)\cdot(\bfxi_{\vec{v}}^n)^{\mathrm{int}}\right|
\\
\leq C \sqrt{\tau} \Vert \partial_t \vec{v}\Vert_{L^2(t^{n-1},t^n;L^\infty(\Omega))} \Vert \bfxi_{\vec{v}}^n\Vert_{\mathrm{DG}}
\Vert \vec{v}\Vert_{L^\infty(0,T;H^1(\Omega))}.
\end{multline*}
Therefore we finally obtain for the term $T_\mathcal{C}^4$:
\[
T_\mathcal{C}^4 \leq r_2 \mu_s K_\strain \Vert \bfxi_{\vec{v}}^n \Vert_{\mathrm{DG}}^2 
+ \frac{C}{r_2} \tau \Vert \partial_t \vec{v}\Vert_{L^2(t^{n-1},t^n;L^\infty(\Omega))}^2.
\]
The bound for $T_\mathcal{C}^5$ is similar but simpler:
\[
T_\mathcal{C}^5 \leq r_2 \mu_s K_\strain \Vert \bfxi_{\vec{v}}^n \Vert_{\mathrm{DG}}^2 
+ \frac{C}{r_2} h^{2q} \Vert \vec{v}\Vert_{L^2(t^{n-1},t^n;H^{q+1}(\Omega))}^2.
\]
Combining the bounds above, we obtain:
\begin{equation*}
T_{11}+T_{12} \leq  4 r_2 \mu_s K_\strain \Vert \bfxi_{\vec{v}}^n \Vert_{\mathrm{DG}}^2 
+ \frac{C}{r_2} \left( \Vert \bfxi_{\vec{v}}^{n-1}\Vert_{L^2(\Omega)}^2
+ h^{2q} + \tau \Vert \partial_t \vec{v}\Vert_{L^2(t^{n-1},t^n;L^\infty(\Omega))}^2
\right).
\end{equation*}
We can rewrite the terms $T_{13}+T_{14}$ as follows
\begin{align*}
T_{13}+T_{14} =& 
- a_{\mathcal{A}}\big(\zeta_c^{n-1},\vec{v}^n,\mathcal{J}(\delta_\tau \xi_c^n)\big)
- a_{\mathcal{A}}\big(\xi_c^{n-1},\vec{v}^n,\mathcal{J}(\delta_\tau \xi_c^n)\big)\\
&- a_{\mathcal{A}}\big(c^n-c^{n-1},\vec{v}^n,\mathcal{J}(\delta_\tau \xi_c^n)\big)
- a_{\mathcal{A}}\big(c_h^{n-1},\bfzeta_{\vec{v}}^n,\mathcal{J}(\delta_\tau \xi_c^n)\big)\\
&- a_{\mathcal{A}}\big(c_h^{n-1},\bfxi_{\vec{v}}^n,\mathcal{J}(\delta_\tau \xi_c^n)\big)
= T_{\mathcal{A}}^1 + \dots + T_\mathcal{A}^5.
\end{align*}
Expanding $T_\mathcal{A}^1$ by definition in \eqref{eq:CHNS:DG_advection}, by Cauchy--Schwarz's inequality, trace inequality, Poincar\'e's inequality, and \cref{lem:CHNS:error_elliptic_projection}, we have
\begin{align*}
\abs{T_\mathcal{A}^1} 
\leq C h^q\norm{c}{L^\infty(0,T;\,H^{q+1}(\Omega))}\norm{\vec{v}^n}{L^\infty(\Omega)}\norm{\mathcal{J}(\delta_\tau \xi_c^n)}{\mathrm{DG}}.
\end{align*}
Using a similar technique as above, we obtain
\begin{align*}
\abs{T_\mathcal{A}^2} &\leq C\norm{\xi_c^{n-1}}{L^2(\Omega)}\norm{\vec{v}^n}{L^\infty(\Omega)}\norm{\mathcal{J}(\delta_\tau \xi_c^n)}{\mathrm{DG}}\\
&\leq C \norm{\xi_c^{n-1}}{\mathrm{DG}} \norm{\vec{v}^n}{L^\infty(\Omega)}\norm{\mathcal{J}(\delta_\tau \xi_c^n)}{\mathrm{DG}}.
\end{align*}
Taking a Taylor expansion of $c$ at $t^{n-1}$ and using similar techniques as above, we have
\begin{equation*}
\abs{T_\mathcal{A}^3} \leq C\tau\norm{\partial_t c}{L^\infty(0,T;L^2(\Omega))}\norm{\vec{v}^n}{L^\infty(\Omega)}\norm{\mathcal{J}(\delta_\tau \xi_c^n)}{\mathrm{DG}}.
\end{equation*}
Again, using similar techniques as above, by the approximation in \cref{lem:Rh}, the mass conservation \cref{thm:CHNS:discrete_mass_conservation}, and the stability bound \cref{eq:CHNS:stability_result1}, we have
\begin{equation*}
\abs{T_\mathcal{A}^4} 
\leq C h^q\norm{\vec{v}^n}{H^{q+1}(\Omega)} \norm{c_h^{n-1}}{L^6(\Omega)}\norm{\mathcal{J}(\delta_\tau \xi_c^n)}{\mathrm{DG}} 
\leq C h^q \norm{\mathcal{J}(\delta_\tau \xi_c^n)}{\mathrm{DG}}.
\end{equation*}
Using the boundedness of $a_{\mathcal{A}}$, the mass conservation \cref{thm:CHNS:discrete_mass_conservation}, the stability bound \cref{eq:CHNS:stability_result1}, and Young's inequality, we have
\begin{align*}
\abs{T_\mathcal{A}^5} 
& \leq C_\gamma\Big( \norm{c_h^{n-1}}{\mathrm{DG}} + \abs{\Omega}\abs{\bar{c}_0}\Big)\norm{\bfxi_{\vec{v}}^n}{L^2(\Omega)}^{1/2}\norm{\bfxi_{\vec{v}}^n}{\mathrm{DG}}^{1/2}\norm{\mathcal{J}(\delta_\tau \xi_c^n)}{\mathrm{DG}}\\
& \leq \frac{C}{r_1^2\,r_2} \norm{\bfxi_{\vec{v}}^n}{L^2(\Omega)}^2 + r_2\mu_s K_\strain \norm{\bfxi_{\vec{v}}^n}{\mathrm{DG}}^2 + \frac{r_1}{5} \norm{\mathcal{J}(\delta_\tau \xi_c^n)}{\mathrm{DG}}^2.
\end{align*}
Therefore, combining the bounds above and using Young's inequality, we obtain
\begin{align*}
T_{13}+T_{14} \leq
&\, \frac{C}{r_1^2\,r_2} \norm{\bfxi_{\vec{v}}^n}{L^2(\Omega)}^2 + r_2\mu_s K_\strain \norm{\bfxi_{\vec{v}}^n}{\mathrm{DG}}^2 \\
&+ r_1\norm{\mathcal{J}(\delta_\tau \xi_c^n)}{\mathrm{DG}}^2 + \frac{C}{r_1}\norm{\xi_c^{n-1}}{\mathrm{DG}}^2  + \frac{C}{r_1}(\tau^2 + h^{2q}).
\end{align*}
For the terms $T_{15}$ and $T_{16}$, by \cref{rem:CHNS:relation_aA_bI}, we may write
\begin{align*}
T_{15} + T_{16} =& 
a_\mathcal{A}(c^n,\bfxi_{\vec{v}}^n,\mu^n) 
- a_\mathcal{A}(c_h^{n-1},\bfxi_{\vec{v}}^n,\mu_h^{n})\\
= & a_\mathcal{A}(\zeta_c^{n-1},\bfxi_{\vec{v}}^n,\mu^n)
+ a_\mathcal{A}(\xi_c^{n-1},\bfxi_{\vec{v}}^n,\mu^n)\\
 &+ a_\mathcal{A}(c^n-c^{n-1},\bfxi_{\vec{v}}^n,\mu^n)
+ a_\mathcal{A}(c_h^{n-1},\bfxi_{\vec{v}}^n,\zeta_\mu^n)
+ a_\mathcal{A}(c_h^{n-1},\bfxi_{\vec{v}}^n,\xi_\mu^n).
\end{align*}
The first three terms simplify since $\mu^n$ does not have any jump. Using Poincar\'{e}'s inequality, we obtain:
\[
|a_\mathcal{A}(\zeta_c^{n-1},\bfxi_{\vec{v}}^n,\mu^n)| \leq \Vert \zeta_c^{n-1}\Vert_{L^2(\Omega)} \Vert \bfxi_{\vec{v}}^n\Vert_{L^4(\Omega)} \Vert \mu^n \Vert_{W^{1,4}(\Omega)}
\leq C h^q \Vert c^{n-1}\Vert_{H^{q+1}(\Omega)} \Vert \bfxi_{\vec{v}}^n\Vert_{\mathrm{DG}},
\]
\[
|a_\mathcal{A}(\xi_c^{n-1},\bfxi_{\vec{v}}^n,\mu^n) | \leq \Vert \xi_c^{n-1}\Vert_{L^2(\Omega)} \Vert \bfxi_{\vec{v}}^n\Vert_{L^4(\Omega)} \Vert \mu^n \Vert_{W^{1,4}(\Omega)}
\leq C \norm{\xi_c^{n-1}}{\mathrm{DG}} \norm{\bfxi_{\vec{v}}^n}{\mathrm{DG}},
\]
and with Taylor expansion
\[
|a_\mathcal{A}(c^n-c^{n-1},\bfxi_{\vec{v}}^n,\mu^n)| \leq \Vert c^n-c^{n-1}\Vert_{L^2(\Omega)} \Vert \bfxi_{\vec{v}}^n\Vert_{L^4(\Omega)}
\Vert \mu^n \Vert_{W^{1,4}(\Omega)}
\leq C \tau \Vert \bfxi_{\vec{v}}^n\Vert_{\mathrm{DG}}.
\]
For the other two terms, by H\"older's inequality, trace inequality, Poincar\'e's inequality, the approximation bound in \cref{lem:CHNS:error_elliptic_projection}, the mass conservation \cref{thm:CHNS:discrete_mass_conservation}, the stability bound \cref{eq:CHNS:stability_result1}, and \cref{lem:CHNS:boundedness_aA}, we have:
\begin{align*}
\abs{a_\mathcal{A}(c_h^{n-1},\bfxi_{\vec{v}}^n,\zeta_\mu^n)} 
&\leq C h^q \norm{\mu}{L^\infty(0,T;H^{q+1}(\Omega))} \norm{c_h^{n-1}}{L^6(\Omega)} \norm{\bfxi_{\vec{v}}^n}{L^3(\Omega)}
\leq C h^q \norm{\bfxi_{\vec{v}}^n}{\mathrm{DG}},
\\
\abs{a_\mathcal{A}(c_h^{n-1},\bfxi_{\vec{v}}^n,\xi_\mu^n)}
&\leq C_\gamma\Big( \norm{c_h^{n-1}}{\mathrm{DG}} + \abs{\Omega}\abs{\bar{c}_0}\Big)\norm{\bfxi_{\vec{v}}^n}{L^2(\Omega)}^{1/2}\norm{\bfxi_{\vec{v}}^n}{\mathrm{DG}}^{1/2}\norm{\xi_\mu^n}{\mathrm{DG}}\\
&\leq \frac{C}{r_2\,r_3^2} \norm{\bfxi_{\vec{v}}^n}{L^2(\Omega)}^2 + \frac{r_2\mu_s K_\strain}{5}\norm{\bfxi_{\vec{v}}^n}{\mathrm{DG}}^2 + r_3 \norm{\xi_\mu^n}{\mathrm{DG}}^2.
\end{align*}
Therefore we obtain:
\begin{align*}
T_{15} + T_{16} \leq&\,
r_2 \mu_s K_\strain \Vert \bfxi_{\vec{v}}^n\Vert_{\mathrm{DG}}^2
+ \frac{C}{r_2} (\tau^2+h^{2q})\\
&+ \frac{C}{r_2} \norm{\xi_c^{n-1}}{\mathrm{DG}}^2
+ \frac{C}{r_2\,r_3^2} \norm{\bfxi_{\vec{v}}^n}{L^2(\Omega)}^2
+ r_3 \norm{\xi_\mu^n}{\mathrm{DG}}^2.
\end{align*}
It remains to find a bound for $\norm{\xi_\mu^n}{\mathrm{DG}}$. We note that $\delta_\tau c^n - (\partial_t c)^n - \delta_\tau\zeta_c^n$ belongs to $M_h$ by taking $\chi = 1$ in \eqref{eq:CHNS:error_equation1}. We choose $\chi = \xi_\mu^n$ in \eqref{eq:CHNS:error_equation1}, use coercivity of $a_\mathcal{D}$, \cref{lem:CHNS:error_property_of_J}, Cauchy--Schwarz's inequality, triangular inequality, and Poincar\'e's inequality to obtain:
\begin{align*}
K_\alpha \norm{\xi_\mu^n}{\mathrm{DG}}^2 
\leq&\, \abs{(\delta_\tau \xi_c^n,\xi_\mu^n)} + \abs{(\delta_\tau c^n - (\partial_t c)^n - \delta_\tau\zeta_c^n, \xi_\mu^n - \frac{1}{\abs{\Omega}}\int_\Omega \xi_\mu^n)} \\
&\,+ \abs{a_\mathcal{A}(c_h^{n-1},\vec{v}_h^n,\xi_\mu^n) - a_\mathcal{A}(c^n,\vec{v}^n,\xi_\mu^n)} \\
\leq&\, C\norm{J(\delta_\tau \xi_c^n)}{\mathrm{DG}}\norm{\xi_\mu^n}{\mathrm{DG}} + C\norm{\delta_\tau c^n - (\partial_t c)^n}{L^2(\Omega)}\norm{\xi_\mu^n}{\mathrm{DG}} \\
&\, + C\norm{\delta_\tau\zeta_c^n}{L^2(\Omega)}\norm{\xi_\mu^n}{\mathrm{DG}} + \abs{a_\mathcal{A}(c_h^{n-1},\vec{v}_h^n,\xi_\mu^n) - a_\mathcal{A}(c^n,\vec{v}^n,\xi_\mu^n)}.
\end{align*}
The last two terms are handled similarly than the terms $T_{13}$ and $T_{14}$. Using Young's inequality, we have
\begin{multline*}
\norm{\xi_\mu^n}{\mathrm{DG}}^2 
\leq C\Big( \norm{\bfxi_\vec{v}^n}{L^2(\Omega)}^2 + \mu_\mathrm{s}K_\strain\norm{\bfxi_\vec{v}^n}{\mathrm{DG}}^2
+ \norm{\xi_c^{n-1}}{\mathrm{DG}}^2 + \tau^2 + h^{2q}\\
+ \norm{J(\delta_\tau\xi_c^n)}{\mathrm{DG}}^2+ \norm{\delta_\tau\zeta_c^n}{L^2(\Omega)}^2
+ \tau\int_{t^{n-1}}^{t^n} \norm{\partial_{tt} c}{L^2(\Omega)}^2 \Big).
\end{multline*}
Next, we remark that $\mathcal{P}_h (\delta_\tau c^n) = \delta_\tau (\mathcal{P}_h c^n)$ and with the approximation result in \cref{lem:CHNS:error_elliptic_projection}, we have
\begin{equation*}
\norm{\delta_\tau\zeta_c^n}{L^2(\Omega)} \leq C h^q \norm{\partial_t c}{L^\infty(0,T;H^{q+1}(\Omega))}.
\end{equation*}
With this bound, the bound for $T_{15}$ and $T_{16}$ becomes:
\begin{align*}
T_{15} + T_{16} 
\leq&\, (r_2+Cr_3) \mu_s K_\strain \norm{\bfxi_{\vec{v}}^n}{\mathrm{DG}}^2 + Cr_3 \norm{J(\delta_\tau \xi_c^n)}{\mathrm{DG}}^2 + C(\frac{1}{r_2\,r_3^2} + r_3) \norm{\bfxi_{\vec{v}}^n}{L^2(\Omega)}^2 \\
&+ C(\frac{1}{r_2}+r_3) \norm{\xi_c^{n-1}}{\mathrm{DG}}^2 + C(\frac{1}{r_2}+r_3) (\tau^2+h^{2q}) + Cr_3 \tau \int_{t^{n-1}}^{t^n}\! \norm{\partial_{tt} c}{L^2(\Omega)}^2.
\end{align*}
To this end, combining \eqref{eq:CHNS:bound_error_lefthand} with all the bounds for $T_1$ to $T_{16}$, and choosing the values $r_1 = K_\alpha/9$, $r_2 = 1/18$, and $r_3 = \min{\{1,K_\alpha\}}/18C$ yields
\begin{align*}
&\frac{K_\alpha}{2}\norm{\mathcal{J}(\delta_\tau \xi_c^n)}{\mathrm{DG}}^2 
+ \frac{\kappa}{2\tau} a_{\mathcal{D}}(\xi_c^n, \xi_c^n) 
- \frac{\kappa}{2\tau} a_{\mathcal{D}}(\xi_c^{n-1}, \xi_c^{n-1}) \\ 
&+ \frac{\mu_\mathrm{s}K_\strain}{2}\norm{\bfxi_{\vec{v}}^n}{\mathrm{DG}}^2 
+ \frac{1}{2\tau}\norm{\bfxi_{\vec{v}}^n}{L^2(\Omega)}^2 
- \frac{1}{2\tau}\norm{\bfxi_{\vec{v}}^{n-1}}{L^2(\Omega)}^2\\
\leq&\, C\big(\norm{\xi_c^n}{\mathrm{DG}}^2 + \norm{\xi_c^{n-1}}{\mathrm{DG}}^2\big) 
+ C\big(\norm{\bfxi_{\vec{v}}^n}{L^2(\Omega)}^2 + \norm{\bfxi_{\vec{v}}^{n-1}}{L^2(\Omega)}^2\big)\\
& + C(\tau^2 + h^{2q})
+ C\frac{h^{2q}}{\tau}\norm{\partial_t \vec{v}}{L^2(t^{n-1},t^n;H^q(\Omega))}^2\\
& + C \tau \int_{t^{n-1}}^{t^n} \big(\norm{\partial_{tt} c}{L^2(\Omega)}^2 + \norm{\partial_t \vec{v}}{L^\infty(\Omega)}^2 + \norm{\partial_{tt} \vec{v}}{L^2(\Omega)}^2\big).
\end{align*}
Multiply by $2\tau$ and sum from $n=1$ to $n=m$, use the coercivity of $a_{\mathcal{D}}$, the fact 
that $\bfxi_{\vec{v}}^0$ is optimally bounded and $\xi_c^0 = 0$:
\begin{align*}
&\tau \sum_{n=1}^m K_\alpha\norm{\mathcal{J}(\delta_\tau \xi_c^n)}{\mathrm{DG}}^2 
+ \kappa K_\alpha \Vert \xi_c^m \Vert_{\mathrm{DG}}^2
+ \tau\sum_{n=1}^m \mu_\mathrm{s}K_\strain\norm{\bfxi_{\vec{v}}^n}{\mathrm{DG}}^2 
+ \norm{\bfxi_{\vec{v}}^m}{L^2(\Omega)}^2 \\
\leq&\, C\tau \sum_{n=1}^m \norm{\xi_c^n}{\mathrm{DG}}^2 
+ C\tau \sum_{n=1}^m \norm{\bfxi_{\vec{v}}^n}{L^2(\Omega)}^2 
+ C(\tau^2 + h^{2q})\\
&+ C \tau^2 \big(\norm{\partial_{tt} c}{L^2(0,T;L^2(\Omega))}^2 + \norm{\partial_t \vec{v}}{L^2(0,T;L^\infty(\Omega))}^2 + \norm{\partial_{tt} \vec{v}}{L^2(0,T;L^2(\Omega))}^2 \big).
\end{align*}
Then, for sufficiently small time step size $\tau$, we can conclude using discrete Gronwall inequality
\begin{align*}
&\tau \sum_{n=1}^m \norm{\mathcal{J}(\delta_\tau \xi_c^n)}{\mathrm{DG}}^2 
+ \Vert \xi_c^m \Vert_{\mathrm{DG}}^2 
+ \tau\sum_{n=1}^m \norm{\bfxi_{\vec{v}}^n}{\mathrm{DG}}^2 
+ \norm{\bfxi_{\vec{v}}^m}{L^2(\Omega)}^2 
\leq C(\tau^2+h^{2}).
\end{align*}
Furthermore it is easy to obtain the following error estimate result
\begin{align*}
\tau \sum_{n=1}^m \norm{\xi_\mu^n}{\mathrm{DG}}^2 \leq C (\tau^2 + h^{2q}).
\end{align*}
\end{proof}

\begin{corollary}
Suppose $(c,\mu,\vec{v},p)$ is a weak solution of \eqref{eq:CHNS:consistency} with regularity \eqref{eq:CHNS:error_regularities}. Then, under \cref{as:CHNS:assumptionA} and sufficiently small time step size $\tau$, there exists a constant $C$ independent of mesh size $h$ and time step size $\tau$ such that for any $m \geq 1$
\begin{multline*}
\max_{1\leq n\leq m}\Big(\norm{c(t^n)-c_h^n }{\mathrm{DG}}^2 
+ \norm{\vec{v}(t^n)-\vec{v}_h^n}{L^2(\Omega)}^2\Big) 
+ \tau\sum_{n=1}^{m}\norm{\vec{v}(t^n)-\vec{v}_h^n}{\mathrm{DG}}^2 \\
+\tau \sum_{n=1}^m \norm{\mu(t^n)-\mu_h^n}{\mathrm{DG}}^2 \leq C (\tau^2 + h^{2q}).
\end{multline*}
\end{corollary}

\section{Conclusions}
In this paper, we have formulated an interior penalty discontinuous Galerkin method for
solving the Cahn--Hilliard--Navier--Stokes equations. The time discretization utilizes a convex-concave
splitting of the chemical energy density and a Picard's linearization for the convection term. Existence and uniqueness of the numerical solution is proved for any general chemical energy density. We show that the discrete total free energy is always dissipative at any time and we obtain stability bounds with any generalized chemical energy density. Under the assumption of a global Lipschitz bound for the chemical energy density, we derive optimal error estimates in time and space. Our analysis of the unique solvability is also valid for non-symmetric versions of the discontinuous Galerkin formulation.

\section*{Acknowledgments}
The authors thank Prof. Vivette Girault for useful discussions and suggestions.

\bibliographystyle{siamplain}
\bibliography{references}
\end{document}

%% file: shared.tex
\usepackage{lipsum}
\usepackage{amsfonts}
\usepackage{graphicx}
\usepackage{epstopdf}
\usepackage{algorithmic}
\ifpdf
  \DeclareGraphicsExtensions{.eps,.pdf,.png,.jpg}
\else
  \DeclareGraphicsExtensions{.eps}
\fi
\usepackage{geometry}  
\usepackage{color}

\usepackage{dsfont} 
  \newcommand{\IN}{\ensuremath\mathds{N}}                        
  \newcommand{\IR}{\ensuremath\mathds{R}}                        
  \newcommand{\IP}{\ensuremath\mathds{P}}                        
\newcommand*{\setE}{\ensuremath{\mathcal{T}}}                    
\newcommand*{\Gammah}{\Gamma_h}                                  
\renewcommand*{\vec}[1]{{\boldsymbol{#1}}}                       
\DeclareMathAlphabet{\mathbfsf}{\encodingdefault}{\sfdefault}{bx}{n}
\newcommand*{\normal}{\vec{n}}                                   
\newcommand*{\dd}{\mathrm{d}}                                    
\newcommand*{\grad}{\nabla}                                      
\renewcommand*{\div}{\nabla\cdot}                                
\newcommand*{\laplace}{\Delta}                                   
\newcommand*{\strain}{{\boldsymbol{\varepsilon}}}                
\newcommand*{\transpose}[1]{{#1}^\mathrm{T}}                     
\newcommand*{\jump}[1]{[#1]}                                     
\newcommand*{\avg}[1]{\{#1\}}                                    
\newcommand*{\abs}[1]{\ensuremath{|#1|}}                         
\newcommand*{\norm}[2]{\|#1\|_{#2}}                              
\newcommand*{\Lnorm}[1]{\|#1\|}                                  
\newcommand*{\on}[2]{\left.#1\right\vert_{#2}}                   

\newsiamremark{remark}{Remark}
\newtheorem{assumption}{Assumption}

\newsiamremark{hypothesis}{Hypothesis}
\crefname{hypothesis}{Hypothesis}{Hypotheses}
\newsiamthm{claim}{Claim}

\headers{Analysis of DG for CHNS system}{Chen Liu and B\'eatrice Rivi\`ere}

\title{Numerical analysis of a discontinuous Galerkin method for Cahn--Hilliard--Navier--Stokes equations}

\author{Chen Liu\thanks{Department of Computational and Applied Mathematics, Rice University, Houston, TX 77005 (\email{chen.liu@rice.edu}, \url{http://cl59.blogs.rice.edu/}).} 
\and B\'eatrice Rivi\`ere\thanks{Department of Computational and Applied Mathematics, Rice University, Houston, TX 77005 (\email{riviere@rice.edu}, \url{http://compm.rice.edu/people-2/beatrice-riviere/}).}}

\usepackage{amsopn}


%% file: paperCHNSanalysis.bbl
\begin{thebibliography}{10}

\bibitem{alpak2016phase}
{\sc F.~O. Alpak, B.~Rivi\`ere, and F.~Frank}, {\em A phase-field method for
  the direct simulation of two-phase flows in pore-scale media using a
  non-equilibrium wetting boundary condition}, Computational Geosciences, 20
  (2016), pp.~881--908, \url{https://doi.org/10.1007/s10596-015-9551-2}.

\bibitem{aristotelous2013mixed}
{\sc A.~C. Aristotelous, O.~Karakashian, and S.~M. Wise}, {\em A mixed
  discontinuous {G}alerkin, convex splitting scheme for a modified
  {C}ahn--{H}illiard equation and an efficient nonlinear multigrid solver},
  Discrete \& Continuous Dynamical Systems-Series B, 18 (2013),
  \url{https://doi.org/10.3934/dcdsb.2013.18.2211}.

\bibitem{Badalassi2003}
{\sc V.~Badalassi, H.~Ceniceros, and S.~Banerjee}, {\em Computation of
  multiphase systems with phase field models}, Journal of Computational
  Physics, 190 (2003), pp.~317--397,
  \url{https://doi.org/10.1016/S0021-9991(03)00280-8}.

\bibitem{barrett1999finite}
{\sc J.~W. Barrett, J.~F. Blowey, and H.~Garcke}, {\em Finite element
  approximation of the {C}ahn--{H}illiard equation with degenerate mobility},
  SIAM Journal on Numerical Analysis, 37 (1999), pp.~286--318,
  \url{https://doi.org/10.1137/S0036142997331669}.

\bibitem{CaiShen18}
{\sc Y.~Cai and J.~Shen}, {\em Error estimates for a fully discretized scheme
  to a {Cahn--Hilliard} phase-field model for two-phase incompressible flows},
  Mathematics of Computation, 87 (2018), pp.~2057--2090,
  \url{https://doi.org/10.1090/mcom/3280}.

\bibitem{ChaabaneGiraultPuelzRiviere2017}
{\sc N.~Chaabane, V.~Girault, C.~Puelz, and B.~Riviere}, {\em Convergence of
  {IPDG} for coupled time-dependent {N}avier--{S}tokes and {D}arcy equations},
  Journal of Computational and Applied Mathematics, 324 (2017), pp.~25--48,
  \url{https://doi.org/10.1016/j.cam.2017.04.002}.

\bibitem{ciarlet2013linear}
{\sc P.~G. Ciarlet}, {\em Linear and nonlinear functional analysis with
  applications}, vol.~130, SIAM, 2013.

\bibitem{diegel2017convergence}
{\sc A.~E. Diegel, C.~Wang, X.~Wang, and S.~M. Wise}, {\em Convergence analysis
  and error estimates for a second order accurate finite element method for the
  {C}ahn--{H}illiard--{N}avier--{S}tokes system}, Numerische Mathematik, 137
  (2017), pp.~495--534, \url{https://doi.org/10.1007/s00211-017-0887-5}.

\bibitem{DingSpelt2007}
{\sc H.~Ding and P.~Spelt}, {\em Wetting condition in diffuse interface
  simulations of contact line motion}, Physical Review E, 75 (2007),
  pp.~\mbox{046708--1}--\mbox{046708--8},
  \url{https://doi.org/10.1103/PhysRevE.75.046708}.

\bibitem{Eyre1998EyreScheme}
{\sc D.~J. Eyre}, {\em Unconditionally gradient stable time marching the
  {C}ahn--{H}illiard equation}, in Symposia BB -- Computational \& Mathematical
  Models of Microstructural Evolution, vol.~529 of MRS Proceedings, 1998,
  \url{https://doi.org/10.1557/PROC-529-39}.

\bibitem{feng2006fully}
{\sc X.~Feng}, {\em Fully discrete finite element approximations of the
  {N}avier--{S}tokes--{C}ahn--{H}illiard diffuse interface model for two-phase
  fluid flows}, SIAM Journal on Numerical Analysis, 44 (2006), pp.~1049--1072,
  \url{https://doi.org/10.1137/050638333}.

\bibitem{feng2007fully}
{\sc X.~Feng and O.~Karakashian}, {\em Fully discrete dynamic mesh
  discontinuous {G}alerkin methods for the {C}ahn--{H}illiard equation of phase
  transition}, Mathematics of Computation, 76 (2007), pp.~1093--1117,
  \url{https://doi.org/10.1090/S0025-5718-07-01985-0}.

\bibitem{FLABR2018Direct}
{\sc F.~Frank, C.~Liu, F.~O. Alpak, S.~Berg, and B.~Rivi\`ere}, {\em Direct
  numerical simulation of flow on pore-scale images using the phase-field
  method}, SPE Journal,  (2018), \url{https://doi.org/10.2118/182607-PA}.

\bibitem{FLAR2018finite}
{\sc F.~Frank, C.~Liu, F.~O. Alpak, and B.~Rivi\`ere}, {\em A finite
  volume/discontinuous {G}alerkin method for the advective {C}ahn--{H}illiard
  equation with degenerate mobility on porous domains stemming from micro-{CT}
  imaging}, Computational Geosciences,  (2018),
  \url{https://doi.org/10.1007/s10596-017-9709-1}.

\bibitem{FLSAR2017energy}
{\sc F.~Frank, C.~Liu, A.~Scanziani, F.~O. Alpak, and B.~Rivi{\`e}re}, {\em An
  energy-based contact angle boundary condition on jagged surfaces for the
  {C}ahn--{H}illiard equation}, Journal of Colloid and Interface Science,
  (2018), \url{https://doi.org/10.1016/j.jcis.2018.02.075}.

\bibitem{girault2017strong}
{\sc V.~Girault, J.~Li, and B.~Rivi{\`e}re}, {\em Strong convergence of the
  discontinuous {G}alerkin scheme for the low regularity miscible displacement
  equations}, Numerical Methods for Partial Differential Equations, 33 (2017),
  pp.~489--513, \url{https://doi.org/10.1002/num.22092}.

\bibitem{GirRav86}
{\sc V.~Girault and P.-A. Raviart}, {\em Finite element methods for
  {N}avier--{S}tokes equations: theory and algorithms}, vol.~5,
  Springer-Verlag, 1986, \url{https://doi.org/10.1007/978-3-642-61623-5}.

\bibitem{girault2005discontinuous}
{\sc V.~Girault, B.~Rivi{\`e}re, and M.~Wheeler}, {\em A discontinuous
  {G}alerkin method with nonoverlapping domain decomposition for the {S}tokes
  and {N}avier--{S}tokes problems}, Mathematics of Computation, 74 (2005),
  pp.~53--84, \url{https://doi.org/10.1090/S0025-5718-04-01652-7}.

\bibitem{girault2005splitting}
{\sc V.~Girault, B.~Rivi{\`e}re, and M.~F. Wheeler}, {\em A splitting method
  using discontinuous {G}alerkin for the transient incompressible
  {N}avier--{S}tokes equations}, ESAIM: Mathematical Modelling and Numerical
  Analysis, 39 (2005), pp.~1115--1147,
  \url{https://doi.org/10.1051/m2an:2005048}.

\bibitem{GuermondMinevShen}
{\sc J.-L. Guermond, P.~Minev, and J.~Shen}, {\em An overview of projection
  methods for incompressible flows}, Computer Methods in Applied Mechanics and
  Engineering, 195 (2006), pp.~6011--6045,
  \url{https://doi.org/10.1016/j.cma.2005.10.010}.

\bibitem{han2015second}
{\sc D.~Han and X.~Wang}, {\em A second order in time, uniquely solvable,
  unconditionally stable numerical scheme for
  {C}ahn--{H}illiard--{N}avier--{S}tokes equation}, Journal of Computational
  Physics, 290 (2015), pp.~139--156,
  \url{https://doi.org/10.1016/j.jcp.2015.02.046}.

\bibitem{heywood1996artificial}
{\sc J.~G. Heywood, R.~Rannacher, and S.~Turek}, {\em Artificial boundaries and
  flux and pressure conditions for the incompressible {N}avier--{S}tokes
  equations}, International Journal for Numerical Methods in Fluids, 22 (1996),
  pp.~325--352,
  \url{https://doi.org/10.1002/(SICI)1097-0363(19960315)22:5<325::AID-FLD307>3.0.CO;2-Y}.

\bibitem{kay2009discontinuous}
{\sc D.~Kay, V.~Styles, and E.~S{\"u}li}, {\em Discontinuous {G}alerkin finite
  element approximation of the {C}ahn--{H}illiard equation with convection},
  SIAM Journal on Numerical Analysis, 47 (2009), pp.~2660--2685,
  \url{https://doi.org/10.1137/080726768}.

\bibitem{kay2008finite}
{\sc D.~Kay, V.~Styles, and R.~Welford}, {\em Finite element approximation of a
  {C}ahn--{H}illiard--{N}avier--{S}tokes system}, Interfaces Free Bound, 10
  (2008), pp.~15--43, \url{https://doi.org/10.4171/IFB/178}.

\bibitem{kay2007efficient}
{\sc D.~Kay and R.~Welford}, {\em Efficient numerical solution of
  {C}ahn--{H}illiard--{N}avier--{S}tokes fluids in 2{D}}, SIAM Journal on
  Scientific Computing, 29 (2007), pp.~2241--2257,
  \url{https://doi.org/10.1137/050648110}.

\bibitem{LFR2018numerical}
{\sc C.~Liu, F.~Frank, and B.~Rivi{\`e}re}, {\em Numerical error analysis for
  non-symmetric interior penalty discontinuous {G}alerkin method of
  {C}ahn--{H}illiard equation}, Numerical Methods for Partial Differential
  Equations,  (2018).

\bibitem{liu2003phase}
{\sc C.~Liu and J.~Shen}, {\em A phase field model for the mixture of two
  incompressible fluids and its approximation by a {F}ourier-spectral method},
  Physica D: Nonlinear Phenomena, 179 (2003), pp.~211--228,
  \url{https://doi.org/10.1016/S0167-2789(03)00030-7}.

\bibitem{novick2008cahn}
{\sc A.~Novick-Cohen}, {\em The {C}ahn--{H}illiard equation}, Handbook of
  Differential Equations: Evolutionary Equations, 4 (2008), pp.~201--228,
  \url{https://doi.org/10.1016/S1874-5717(08)00004-2}.

\bibitem{riviere2008}
{\sc B.~Rivi\`ere}, {\em Discontinuous {G}alerkin Methods for Solving Elliptic
  and Parabolic Equations: Theory and Implementation}, Frontiers in Applied
  Mathematics, Society for Industrial and Applied Mathematics, 2008,
  \url{https://doi.org/10.1137/1.9780898717440}.

\bibitem{shen2015decoupled}
{\sc J.~Shen and X.~Yang}, {\em Decoupled, energy stable schemes for
  phase-field models of two-phase incompressible flows}, SIAM Journal on
  Numerical Analysis, 53 (2015), pp.~279--296,
  \url{https://doi.org/10.1137/140971154}.

\bibitem{temam2001navier}
{\sc R.~Temam}, {\em {N}avier--{S}tokes equations: theory and numerical
  analysis}, vol.~343, American Mathematical Soc., 2001,
  \url{https://doi.org/10.1090/chel/343}.

\bibitem{yang2017linear}
{\sc X.~Yang and J.~Zhao}, {\em On linear and unconditionally energy stable
  algorithms for variable mobility {C}ahn--{H}illiard type equation with
  logarithmic {F}lory--{H}uggins potential}, arXiv preprint arXiv:1701.07410,
  (2017).

\bibitem{YuilleRangarajan2003}
{\sc A.~L. Yuille and A.~Rangarajan}, {\em The concave-convex procedure},
  Neural Computation, 15 (2003), pp.~915--936,
  \url{https://doi.org/10.1162/08997660360581958}.

\end{thebibliography}
